\documentclass[11pt,reqno]{amsart}

\usepackage[a4paper, left=30mm,right=30mm,top=30mm,bottom=30mm,marginpar=25mm]
{geometry}
\usepackage[english]{babel}
\usepackage[utf8x]{inputenc}
\usepackage{setspace}
\usepackage{enumitem} 
\usepackage{amsmath}
\usepackage{mathtools}
\usepackage[normalem]{ulem}
\usepackage{amsfonts}
\usepackage{amssymb}
\usepackage{fancyhdr}
\usepackage{textcomp}
\usepackage{graphicx} 
\usepackage{float}
\usepackage{listings}
\usepackage{amssymb}
\usepackage{color}

\newtheorem{theorem}{Theorem}[section]
\newtheorem{lemma}{Lemma}[section]

\newtheorem{definition}{Definition}[section]

\newtheorem{remark}{Remark}[section]

\newcommand{\be}{\begin{equation}}
\newcommand{\ee}{\end{equation}}
\newcommand{\del}{\partial}
\newcommand{\e}{\varepsilon}
\newcommand{\eps}{\varepsilon}
\newcommand{\To}{{\mathbb{T}}}

\def\R{\mathbb{R}}

\definecolor{listinggray}{gray}{0.9}
\definecolor{lbcolor}{rgb}{0.9,0.9,0.9}

\allowdisplaybreaks
\def\del{\partial}
\def\bu{\bar u}

\def\Z{\mathbb Z}
\begin{document}
	
	\numberwithin{equation}{section}
	
	\title[Relative Entropy Method for Inhomogeneous Systems of Balance Laws]{The Relative Entropy Method for Inhomogeneous Systems of Balance Laws}

	\author[C. Christoforou]{Cleopatra Christoforou}
	\address[Cleopatra Christoforou]{Department of Mathematics and Statistics,
		University of Cyprus, Nicosia 1678, Cyprus.}
	\email{christoforou.cleopatra@ucy.ac.cy} 
	
	\date{\today}
	
	\begin{abstract}
	General hyperbolic systems of balance laws with inhomogeneity in space and time in all constitutive functions are studied in the context of relative entropy. A framework is developed in this setting that contributes to a measure-valued weak vs strong uniqueness theorem, a stability theorem of viscous solutions and a convergence theorem as the viscosity parameter tends to zero. The main goal of this paper is to develop hypotheses under which the relative entropy framework can still be applied. Examples of systems with inhomogeneity that have different charateristics are presented and the hypotheses are discussed in the setting of each example.
	\end{abstract}
	
	\maketitle 
	
\section{Introduction}\label{sec1}
	
	General systems of conservation laws in several space dimensions 
	\be\label{Sec1:Uorig-eq}
\del_t U+\del_\alpha(f_\alpha(U))=0
\ee
	is the subject of understanding wide-ranging phenomena that describe a broad class of partial differential equations in Continuum Physics. Here $U\in\mathbb{R}^n$ is a function of $(x,t)\in\mathbb{R}^d\times\mathbb{R}_+$ and $f_\alpha:\mathbb{R}^n\to\mathbb{R}^n$, $\alpha=1,\dots d$ are given smooth fluxes. The necessity to move from the particular to the general has been a tactic that produced many research results and gave answers to a class of problems simultaneously. In this spirit, it is the motivation of the work in this article. The aim is to capture the relative entropy method for systems that belong to the general class of inhomogeneous balance laws
	\be\label{Sec1:Ueq}
\del_t (A(U,x,t))+\del_\alpha(f_\alpha(U,x,t))+P(U,x,t)=0\;,
\ee
with initial data
\be\label{sec1:Udata}
U(x,0)=U_0(x)\;.
\ee
Here, the constitutive function $A:\mathbb{R}^n\times\mathbb{R}^d\times \mathbb{R}_+\to\mathbb{R}^n$ and $f_\alpha,\,P:\mathbb{R}^n\times\mathbb{R}^d\times \mathbb{R}_+\to\mathbb{R}^n$ that correspond to the fluxes and the source, respectively, depend explicitly on $(x,t)$ and the presence of this inhomogeneity is the main component investigated in this article in conjuction with the other challenges that systems of balance laws have. There are models in the literature that belong to this class~\eqref{Sec1:Ueq} of systems and the  explicit dependence is an important characteristic of the phenomena described. 
Such examples include the flow of a a gas through a duct of varying cross section, models with nonlocal terms corresponding to memory effect.
In general, explicit dependence of $A$ and $f_\alpha$ on time $t$ that indicates "ageing" of the medium is quite rare in Continuum Mechanics. However, it is quite more common to see the source $P$ to depend explicitly on time and this corresponds to a time-dependent forcing. On the other hand, examples coming from isometric immersions are characterized by such features and inhomogeneity (in both independent variables) is present even in the constitutive functions $A$ and $f_\alpha$.
	
	 The relative entropy method was introduced by Dafermos \cite{dafermos79,dafermos79b} and DiPerna \cite{diperna79} and it is a quite powerful technique in comparing solutions of one or more conservation laws. At the early stages of its development, the relative entropy method captured uniqueness and stability results in the 
hyperbolic context \eqref{Sec1:Uorig-eq} and in the sequel, the method has been flourished in various directions. For instance, it has been used not only for conservation laws  ({\it e.g.}~\cite{diperna79,bds11,dst12,SV14}), or balance laws ({\it e.g.}~\cite{tzavaras05,MT14}), but also hyperbolic-parabolic systems  ({\it e.g}~\cite{fn12,lt06,lt13,KV15}). By construction, it is evident its connection with thermodynamics and over the years, applications in other settings have been studied. The power of the method and the plethora of results on special systems motivated the work of Christoforou-Tzavaras~\cite{christoforou2016relative} that studied it in a general setting. More precisely, in~\cite{christoforou2016relative}, hyperbolic-parabolic systems written in the general form
\be\label{Sec1:UCT-eq}
\del_t (A(U))+\del_\alpha(f_\alpha(U))=\e\del_\alpha (B_{\alpha\beta}(U)\del_\alpha U)\;
\ee
are considered, with viscosity matrices $B_{\alpha\beta}$, $\alpha,\beta=1\dots d$, and hypotheses are assumed so that the relative entropy method is performed. Having the relative entropy identity, stability of viscous solutions with respect to initial data is established as well as convergence as $\e\to0+$ of viscous solutions to the smooth solution of the hyperbolic system. Also, weak-strong uniqueness in the hyperbolic regime is obtained within the class of dissipative measure-valued solutions. The aim in~\cite{christoforou2016relative} is to include systems, that are hyperbolic-parabolic with $A(U)$ not necessarily being equal to $U$, under the machinery of the relative entropy. Applications in thermoviscoelasticity are studied in this context.

	 The objective of this article is to systematize the derivation of relative entropy identities for the inhomogeneous systems~\eqref{Sec1:Ueq}. The idea is similar to the one employed in~\cite{christoforou2016relative}, but here, we allow the presence of inhomogeneity in space $x\in\mathbb{R}^d$ and time $t>0$ in the constitutive functions  of the hyperbolic systems. In other words, we assume that $A$, $f_\alpha$ and $P$ depend explicitly on $(x,t)$ and investigate the hypotheses needed to perform the relative entropy method for system~\eqref{Sec1:Ueq} while in~\cite{christoforou2016relative} dependence only on the state $U$ was considered. 

There is a lot of effort in the community of conservation laws to identify the appropriate weak framework for systems of conservation laws  in more than one space dimension in which global existence of weak solutions can be established. For the time being, the well-posedness for systems of several space dimensions is an unexplored area with significant potential. However, in one-space dimension, systems of conservation laws~\eqref{Sec1:Uorig-eq} are well studied and global existence of entropy weak solutions is established for small initial data of bounded variation. Inhomogeneous systems~\eqref{Sec1:Ueq} in one space dimension, $d=1$, have been studied under appropriate dissipativeness conditions on the source $P$ and compactness properties of the solution $U$ and its derivatives. These could be achieved under appropriate decay rates of the constitutive functions with respect to their dependence in space $x$ and time $t$. Techniques that have been developed in one-space dimension for systems~\eqref{Sec1:Ueq} to construct approximate solutions that converge (up to a subsequence) such as the random choice method, the front tracking algorithm, the vanishing viscosity and the compensated compactness method have been extended in the setting of~\eqref{Sec1:Ueq}. Dafermos and Hsiao~\cite{DH} first established global existence when $A(U,x,t)\equiv U$ under appropriate conditions on the derivates of the flux and the source on $(x,t)$ using random choice method. Other related results on global existence for inhomogeneous systems are~\cite{AG, C0}. An exposition of the current state of the theory can also be found in the book~\cite{MR3468916}  by Dafermos. In short, we study the relative entropy framework of~\eqref{Sec1:Ueq} by combining ideas from~\cite{christoforou2016relative} having in mind the existing results for inhomogeneous systems~\eqref{Sec1:Ueq} in one-space dimension. 

Further analysis is exploited in the setting of hyperbolic-parabolic systems with inhomogeneity
\be\label{Sec1:Ueqpar}
\del_t (A(U,x,t))+\del_\alpha(f_\alpha(U,x,t))+P(U,x,t)=\e \del_\alpha(B_{\alpha\beta}(U,x,t)\del_\beta U)
\ee
with $B_{\alpha\beta}:\mathbb{R}^n\times\mathbb{R}^d\times[0,\infty)\to\mathbb{R}^{n\times n}$ given smooth functions, $\alpha,\beta=1,\dots,d$.
The dependence on $(x,t)$ is also taken into account in the viscosity matrix $B$. The relative entropy framework is worked out to produce a relative entropy inequality under additional hypothesis on the viscosity matrices $B_{\alpha\beta}$ as studied in~\cite{christoforou2016relative}. Again, further conditions are imposed to control the dependence of $B_{\alpha\beta}$ on $(x,t)$.

The structure of the paper is the following: In Section~\ref{sec2}, we state the hypotheses under which the relative entropy method is performed. In Section~\ref{sec3}, we perform the computations and arrive at the relative entropy identity for~\eqref{Sec1:Ueq} and \eqref{Sec1:Ueqpar} in each subsection. In Section~\ref{sec4}, we state and prove the theorems: In Subsection~\ref{S4.1}, we exploit the definition of dissipative measure-valued solutions and prove the weak-strong uniqueness result for~\eqref{Sec1:Ueq}; in Subsection~\ref{S4.2}, we obtain the $L^2$ stability of viscous solutions to~\eqref{Sec1:Ueqpar} while in Subsection~\ref{S4.3}, we establish the convergence to a smooth solution of~\eqref{Sec1:Ueq} as $\e\to0+$. Section~\ref{sec6} includes examples that belong to the hyperbolic class of inhomogeneous systems~\eqref{Sec1:Ueq} and the hyperbolic-parabolic one~\eqref{Sec1:Ueqpar}. In a forthcoming paper of the author, an application of different flavor is presented; this is the isometric immersion problem into $\mathbb{R}^3$ and the presence of inhomogeneity there is crucial part of the issue. This was actually the motivation of the author that resulted to the present article.

\section{Hypotheses and Preliminaries} \label{sec2}

In this section, we consider a weak solution $U$ to the inhomogeneous system
\be\label{Sec3:Ueq}
\del_t (A(U,x,t))+\del_\alpha(f_\alpha(U,x,t))+P(U,x,t)=0
\ee
and a strong solution $\bar U$ to
\be\label{Sec3:barUeq}
\del_t (A(\bar U,x,t))+\del_\alpha(f_\alpha(\bar U,x,t))+P(\bar U,x,t)=0
\ee
and set appropriate hypotheses on the constitutive functions and $L^p$-type growth conditions, that allow us to derive the relative entropy identity between $U$ and $\bar U$ in the next section. Towards the end of this section, we include additional hypotheses needed to treat the hyperbolic-parabolic case.

To begin with, we give some useful notation.

\noindent
{\bf Notation.}\emph{
Let $\xi=\xi(U,x,t)$ be a generic function of $(U,x,t)$. Then $\xi_t$ and $\xi_{x_\alpha}$ denote the partial derivatives of $\xi(U,x,t)$ with respect to the components $t$ and $x_\alpha$ respectively. Also, $\nabla \xi:=\nabla_U\xi$ is the gradient of $\xi=\xi(U,x,t)$ with respect to  the vector $U$. For the case that $\psi$ is a function of $(x,t)$, $\psi=\psi(x,t)$, the partial derivatives of $\psi$ with respect to $x_\alpha$ and $t$ are denoted by $\partial_\alpha \psi$ and $\partial_t \psi$. Hence, for the case that $U$ is a function of $(x,t)$, i.e. $U=U(x,t)$,
we have immediately that the partial derivatives of  $\xi(U(x,t),x,t)$ w.r.t. $x_\alpha$ and $t$ satisfy the expressions
\be
\partial_{x_\alpha} \xi(U(x,t),x,t)= \nabla\xi(U(x,t),x,t)\, \partial_\alpha U(x,t)+\xi_{x_\alpha}(U(x,t),x,t)
\ee
For convenience, from here and on, we write $\xi=\xi(U,x,t)$ having in mind that $U=U(x,t)$.\\
We also adopt the standard summation notation throughout the paper and from here and on, we use the abbreviation $\xi=\xi(U,x,t)$ and $\bar \xi=\xi(\bar U,x,t)$ for convenience.
}

It should be mentioned that that in what follows, we consider either the whole space $\mathbb{R}^d\times[0,T]$ and assume that the solutions decay as $|x|\to\infty$ or the domain $Q_T=  \To^d\times[0,T]$ for periodic solutions with $\To^d =(\R/2\pi\Z)^d$ and $T>0$ a finite time of existence.  

Next, we state hypotheses on system~\eqref{Sec3:Ueq} and provide some useful remarks.  All hypotheses are divided into two classes, called Hypotheses {\bf A} and {\bf B}. The first class of five hypotheses are needed despite of the inhomogeneity and can be found also in~\cite{christoforou2016relative} but here are written with the presence of $(x,t)$ dependence. See also Gwiazda et al~\cite{GOA} for a further generalization of the hypotheses set in~\cite{christoforou2016relative}.
The second class of three hypotheses are implemented to treat the presence of dependence on $(x,t)$. There is another class, called Hypotheses {\bf C},  that consists of two hypotheses and they reflect the diffusion part when dealing with the hyperbolic-parabolic system~\eqref{Sec1:Ueqpar}. 

Here it is the first class of hypotheses:\\
\noindent
{\bf Hypotheses A.} We assume the following hypotheses:\\
{\bf ($\textbf{H}_1$)} At every point $(x,t)$, the map $A(\cdot,x,t):\mathbb{R}^n\to\mathbb{R}^n$ is a $C^2$ map, that satisfies that $\nabla A(U,x,t)$ is nonsingular $\forall$ $U\in\mathbb{R}^n$.\\
{\bf ($\textbf{H}_2$)} At every point $(x,t)$, there exist an entropy-entropy flux pair $(\eta(\cdot,x,t), q(\cdot,x,t))$, i.e. there exists  a smooth function $G(\cdot,x,t):\mathbb{R}^n\to\mathbb{R}^n$, $G(\cdot,x,t)=G(U,x,t)$, such that  
\begin{equation}\tag{$\text{H}_2$}\label{S2:H2}
\begin{aligned}
\nabla \eta(U,x,t)&=G(U,x,t)\cdot\nabla A(U,x,t)\\
\nabla q_\alpha(U,x,t)&=G(U,x,t)\cdot\nabla f_\alpha(U,x,t),\qquad\alpha=1,\dots,d\,.
\end{aligned}
\end{equation}
{\bf ($\textbf{H}_3$)} At every point $(x,t)$, the symmetric matrix $$\nabla^2 \eta(U,x,t)- G(U,x,t)\cdot\nabla^2A(U,x,t)$$ is strictly positive definite uniformly in $x$ and $t$, i.e. there exists a positive constant $\mu$ independent of $(x,t)$ such that 
\be\tag{$\text{H}_3$}\label{S2:H3}
\xi^T\left( \nabla^2 \eta(U,x,t)- G(U,x,t)\cdot\nabla^2A(U,x,t) \right)\xi\ge \mu |\xi|^2>0\qquad \forall\xi\in\mathbb{R}^n\setminus\{ 0\}.
\ee
{\bf ($\textbf{H}^{B}$)} We assume boundedness of the constitutive functions w.r.t. the inhomogeneity. In other words, when $U$ takes values in a bounded ball $B_M\subset\mathbb{R}^n$ centered at the origin with radius $M>0$, then
\be\tag{$\text{H}^{B}$}\label{S2:HB}
\begin{aligned}
&|A(U,x,t)|+ |\nabla A(U,x,t)|+|f_\alpha (U,x,t)|+|\nabla f_\alpha (U,x,t)|\le C,\\
&|G(U,x,t)|+|\nabla G(U,x,t)|\le C
\end{aligned}
\ee
$\forall U\in B_M, \,x\in\mathbb{R}^d,\, t>0$, for some constant $C$ possibly depending on $M$.\\
{\bf ($\textbf{H}^{gr}$)} There are positive constants $\beta_1$, $\beta_2$, $\beta_3$ such that
\be\tag{$\text{H}^{gr}_1$}\label{S2:Hgr1}
\beta_1(|U|^p+1)-\beta_3\le\eta(U,x,t)\le\beta_2(|U|^p+1),\qquad\forall U\in\mathbb{R}^n, \,x\in\mathbb{R}^d,\, t>0
\ee
and for every $(x,t)$, it holds
\be\tag{$\text{H}^{gr}_2$}\label{S2:Hgr2}
\frac{|f_\alpha(U,x,t)|}{\eta(U,x,t)}=o(1) \qquad\text{as }|U|\to\infty,
\ee
\be\tag{$\text{H}^{gr}_3$}\label{S2:Hgr3}
\frac{|A(U,x,t)|}{\eta(U,x,t)}=o(1) \qquad\text{as }|U|\to\infty\,.
\ee

\begin{remark}
By hypothesis { $(\text{H}_1)$}, for every $(x,t)$,  the map $U\mapsto V=A(U,x,t)$ is globally invertible with the inverse map $U=A^{-1}(V,x,t)$ to be a $C^2$ map. Moreover, 
one can verify that hypothesis \eqref{S2:H2} holds true if the multiplier $G$ satisfies simultaneously the equations
\begin{align}
\nabla G(U,x,t)^T\nabla A(U,x,t)&=\nabla A(U,x,t)^T\nabla G(U,x,t)\\
\nabla f_\alpha(U,x,t)^T\nabla A(U,x,t)&=\nabla f_\alpha(U,x,t)^T\nabla G(U,x,t)\qquad\alpha=1,\dots,d,
\end{align}
for every $x$ and $t$. 
\end{remark}

By hypotheses {$(\text{H}_{1})$}--\eqref{S2:H2}, we have  $G=G(U,x,t)=\nabla\eta\cdot(\nabla A)^{-1}$. Hence, using the relations,
$$\del_t\eta(U,x,t)=\nabla \eta \partial_t U+\eta_t, $$
$$\del_\alpha q_\alpha(U,x,t)=\nabla q_\alpha \del_\alpha U+q_{\alpha,x_\alpha}, $$
and multiplying~\eqref{Sec3:Ueq} by $G$, we arrive at the entropy relation
\be\label{Sec3:etaeq}
\del_t(\eta(U,x,t))+\del_\alpha q_\alpha(U,x,t) +G\cdot R\le \eta_t+q_{\alpha,x_\alpha}
\ee
where $R:=P+A_t+f_{\alpha,x_\alpha}$ and $Z:=G\cdot R-\eta_t-q_{\alpha,x_\alpha}$, that is an inequality that holds true in the distribution sense for the class of weak solutions.

Following the analysis in~\cite{christoforou2016relative}, we immediately see that system~\eqref{Sec3:Ueq} can be expressed as
\begin{align}
\label{sys2}
\del_t V + \del_\alpha \tilde{f_\alpha} (V,x,t) + \tilde{P}(V,x,t)& = 0 
\end{align}
written in the conserved variable $V = A(U,x,t)$ and in this form, it attains a convex entropy $\tilde{\eta}(V,x,t)$ under hypothesis \eqref{S2:H3} that satisfies
\begin{align}
\del_t \tilde{\eta}  (V,x,t)  + \del_\alpha  \tilde{q_\alpha} (V,x,t)   +\tilde{Z}(V,x,t)= 0\;.
\label{ensys2}
\end{align}
Here, $\tilde{g}(V,x,t)$ associated with a generic function $g(\cdot,x,t)$ is given via the relation $$g(U,x,t)=g(A^{-1}(V,x,t),x,t):=\tilde{g}(V,x,t),$$
where $A^{-1}(\cdot,x,t)$ denotes the inverse of the map $A(\cdot,x,t)$ according to hypothesis {$(\text{H}_{1})$}.
Indeed, we have that the multiplier is
\begin{align}
\label{appform3}
G(U,x,t) =  ( \nabla_V \tilde{\eta} ) (A(U,x,t),x,t)  
\end{align}
and the identity
\begin{equation}\label{appform2}
\begin{aligned}
\nabla^2_U \eta (U,x,t) - (\nabla_V \tilde{\eta} )& (A(U,x,t),x,t) \cdot \nabla^2_U A(U,x,t)=  \\
&=(\nabla_V^2 \tilde{\eta})(A(U,x,t),x,t) : ( \nabla_U A(U,x,t) , \nabla_U A(U,x,t) )  \;.
\end{aligned}
\end{equation}
implies that hypothesis \eqref{S2:H3}  translates 
to the requirement  that the entropy  $\tilde{\eta}(V,x,t)$ is convex in $V$, i.e.
$$
\zeta \cdot \nabla_V^2 \tilde{\eta} (V,x,t)  \zeta > 0  \quad \mbox{for $\zeta \in \R^n, \; \zeta \ne 0$}  \, ,
$$
uniformly in $x$ and $t$.

Now, we define the relative entropy quantity $\eta(U|\bar U; x,t)$ between the weak solution $U$ and the strong solution $\bar U$ to be
\be\label{Sec3:relative eta}
\eta(U|\bar U; x,t):=\eta(U,x,t)-\eta(\bar U,x,t)-G(\bar U,x,t)\left( A(U,x,t)-A(\bar U,x,t)\right)
\ee
and the corresponding relative entropy fluxes
\be\label{Sec3:relative q}
q_\alpha(U|\bar U; x,t):=q_\alpha(U,x,t)-q_\alpha(\bar U,x,t)-G(\bar U,x,t)\left( f_\alpha(U,x,t)-f_\alpha(\bar U,x,t)\right)
\ee
for $\alpha=1,\dots, d$. For simplification, we use the abbreviation $\xi=\xi(U,x,t)$ and $\bar \xi=\xi(\bar U,x,t)$ as described in the paragraph presenting notation. 
We also need the relative multiplier
\be\label{Sec3:relative G}
G(U|\bar U; x,t):=G(U,x,t)-\bar G-\nabla \bar G (\nabla \bar A)^{-1} \left( A(U,x,t)-\bar A\right)\;,
\ee
and the relative fluxes
\be\label{Sec3:relative f}
f_\alpha(U|\bar U;x,t):=f_\alpha(U,x,t)-\bar f_\alpha-\nabla \bar f_\alpha(\nabla \bar A)^{-1} \left( A(U,x,t)-\bar A \right)\;.
\ee
Let us also add that the relative quantities defined above coincide with those in~\cite{christoforou2016relative}, but here they also depend explicitly on $(x,t)$ due to inhomogeneity.

\noindent
{\bf Hypotheses B.} We continue now with the set of hypotheses due to the inhomogeneity.\\
{\bf ($\text{H}^{x,t}$)} If $U$ and $\bar U$ take values in a bounded ball $B_M\subset\mathbb{R}^n$ centered at the origin with radius $M>0$, then there exists a constant $C$ possibly depending on $M$ such that
\begin{align}\label{S2:Hxt}
&|A(U,x,t)-\bar A| + |f_\alpha(U,x,t)-\bar f_\alpha| \,\le C|U-\bar U|^2  \tag{$\text{H}^{x,t}_1$}\\
&|A_t(U,x,t)-\bar A_t+f_{\alpha,x_\alpha}(U,x,t)-\bar f_{\alpha,x_\alpha}| \le C|U-\bar U|^2  \tag{$\text{H}^{x,t}_2$}\label{S2:Hxt2}\\
&|\eta_t(U,x,t)-\bar \eta_t+q_{\alpha,x_\alpha}-\bar q_{\alpha,x_\alpha}|
\le C|U-\bar U|^2  \tag{$\text{H}^{x,t}_3$}\label{S2:Hxt3}
\end{align}
for all $x\in\mathbb{R}^d$, $t>0$.\\
{\bf ($\textbf{H}_{x,t}^B$)} We assume boundedness of the constitutive functions w.r.t. the inhomogeneity. In other words, when $U$ takes values in a bounded ball $B_M\subset\mathbb{R}^n$ centered at the origin with radius $M>0$, then
\be\tag{$\text{H}^{B}_{x,t}$}\label{S2:HBxt}
\begin{aligned}
&|R(U,x,t)|+ |G_t(U,x,t)|+|G_{x_\alpha} (U,x,t)|\le C,
\end{aligned}
\ee
$\forall U\in B_M, \,x\in\mathbb{R}^d,\, t>0$, for some constant $C$ possibly depending on $M$.\\
{\bf ($\textbf{H}^{gr}_{x,t}$)} We assume growth conditions on the constitutive functions that are present due to inhomogeneity. More precisely, for every $(x,t)$, it holds
\be\tag{$\text{H}^{gr}_4$}\label{S2:Hgr4}
\frac{|\eta_t(U,x,t)+q_{\alpha,x_\alpha}(U,x,t)|}{\eta(U,x,t)}=o(1) \qquad\text{as }|U|\to\infty,
\ee
\be\tag{$\text{H}^{gr}_5$}\label{S2:Hgr5}
\frac{|A_t(U,x,t) + f_{\alpha,x_\alpha}(U,x,t) |}{\eta(U,x,t)}=o(1) \qquad\text{as }|U|\to\infty\,.
\ee
{\bf ($\textbf{H}^{R}$)} Either \be\tag{$\text{H}^{R}_{1}$}\label{S2:HR1}
\frac{|G(U,x,t) \cdot R(U,x,t)|}{\eta(U,x,t)}=o(1),\qquad \frac{|R(U,x,t)|}{\eta(U,x,t)}=o(1)\quad\text{as } |U|\to\infty,
\ee
hold true, for every $x\in\mathbb{R}^d$, $t>0$\\
or
\be\tag{$\text{H}^{R}_{2}$}\label{S2:HR2}
(G(U,x,t)-\bar G)\cdot(R(U,x,t)-\bar R)\ge 0
\ee
for all $U,\bar U\in\mathbb{R}^n$, $x\in\mathbb{R}^d$, $t>0$.
%
Moroever, for every $(x,t)$ it holds
\be\tag{$\text{H}^{R}_{3}$}\label{S2:HG}
\frac{|G(U,x,t)|}{\eta(U,x,t)}=o(1) \qquad\text{as }|U|\to\infty\;.
\ee

\begin{remark} Let us place some remarks regarding Hypotheses B.\\
1. 
In the existing theory for $1$-d inhomogeneous systems, assumptions are set so that the bound for terms similar to those in~\eqref{S2:Hxt}--\eqref{S2:Hxt3} consists of a function $\psi(x,t)$ that belongs to $L^1(\mathbb{R}\times[0,\infty)$. The purpose of this is to achieve apriori bounds on the total variation and prove compactness of the approximate sequence constructed via the random choice method in conjuction with operator splitting (cf.~\cite{DH}). The $BV$ estimates are of $L^1$ type and therefore such a condition works well in that setting. Here, the aim is to apply the relative entropy method, which by construction is of $L^2$ type. This explains the need to have bounds of the form~\eqref{S2:Hxt}--\eqref{S2:Hxt3}. However, in both cases, the heart of the matter is that these terms are almost negligible as $|x|\to\infty$ or $t\to+\infty$ and as a result the inhomogeneity is absent there.\\
2. Assumptions~\eqref{S2:HR1}--~\eqref{S2:HG} are needed to control the terms
$$(G(U,x,t)-\bar G)\cdot(R(U,x,t)-\bar R),\qquad G(U|\bar U,x,t)$$
in the relative entropy identity. Condition~\eqref{S2:HR2} corresponds to a dissipative source term.
\end{remark}
%
\noindent
{\bf Hypotheses C.} We conclude this section with the additional hypotheses needed to treat systems that are hyperbolic-parabolic. Consider the inhomogeneous system with diffusion
\be\label{Sec2:Ueqpar}
\del_t (A(U,x,t))+\del_\alpha(f_\alpha(U,x,t))+P(U,x,t)=\e \del_\alpha(B_{\alpha\beta}(U,x,t)\del_\beta U)\;.
\ee Then the evolution of the entropy will obey the relation
\begin{equation}
\label{hypparaen}
\begin{aligned}
\del_t  \eta (U,x,t) + \del_\alpha q_\alpha (U,x,t) +G\,\cdot R &=  \eta_t+q_{\alpha,x_\alpha}+\eps \del_\alpha ( G(U,x,t) \cdot B_{\alpha \beta} (U,x,t) \del_\beta U) \\
&- \eps (\nabla G(U,x,t) \del_{\alpha} U+G_{x_\alpha}) \cdot B_{\alpha \beta} (U,x,t) \del_{\beta} U \, 
\end{aligned}
\end{equation}
with the viscosity $\e>0$ term present in the the diffusion terms. In Subsection~\ref{S4.2}, we prove stability of viscous solutions to~\eqref{Sec2:Ueqpar} and in Subsection~\ref{S4.3}, we show convergence as $\eps \to 0$ of viscous approximations to the solution of the hyperbolic system~\eqref{Sec3:Ueq}. To establish these two results, we assume part of Hypotheses A and B imposed on the hyperbolic part and in addition to that we add one of the two following hypotheses depending which result we exploit.

\noindent
{\bf ({H$_{P_1}$)}} 
Equation \eqref{hypparaen} admits a dissipative structure, namely that the following strictly positive definite structure holds true:
\begin{equation}
\sum_{\alpha,\beta=1}^d\xi_\alpha \cdot \left(\nabla G(U,x,t)^T B_{\alpha \beta} (U,x,t) \, \xi_\beta\right) =\sum_{\alpha,\beta=1}^d\sum_{i,j=1}^n\xi_\alpha^i D_{\alpha\beta}^{ij} \xi_\beta^j  > 0  
\end{equation}
for all  
$\xi=(\xi_1,\dots,\xi_d) \in \R^{d \times n} \, , \xi \neq 0$ uniformly in $x$ and $t$. Here $D_{\alpha\beta}\doteq\nabla G^T B_{\alpha\beta}  $. We can easily get that, for solutions $U$ that decay to zero as $|x|\to\infty$, there is a positive constant $\lambda_1$, possibly depending on the radius $M$ of a ball $B_M$ within which $U$ takes values, but independent of $x$ and $t$ such that 
\begin{equation}\label{S2:HP1}
\sum_{\alpha,\beta=1}^d\xi_\alpha \cdot \left(\nabla G(U,x,t)^T B_{\alpha \beta} (U,x,t) \, \xi_\beta\right) \ge \lambda_1|\xi|^2 \;
\tag {H$_{P_1}$}
\end{equation}
with $|\cdot|$ to be the Euclidean norm in $\mathbb{R}^{n\times d}$.
This condition guarantees that the entropy dissipates along the evolution. Also, it is a natural condition in the context of applications to mechanics as it is connected to the Clausius-Duhem inequality.

\noindent
 {\bf ({H$_{P_2}$)}} There is a positive constant $\lambda_2$ independent of $x$ and $t$ such that the viscosity matrices satisfy
\begin{equation}
\label{S2:HP2}
\sum_{\alpha,\beta}\nabla G(U,x,t)\del_\alpha u\cdot B_{\alpha\beta}(U,x,t)\del_\beta u\ge \lambda_2\sum_\alpha| \sum_\beta B_{\alpha\beta}(U,x,t)\del_\beta U|^2\;.
\tag {H$_{P_2}$}
\end{equation}
This condition allows degenerate viscosity matrices to be considered in the zero-viscosity limit.

Except of one of the two hypotheses above, we assume boundedness of the viscosity matrix w.r.t. the inhomogeneity; i.e.\\
\noindent
{\bf ($\textbf{H}^{B}_{{par}}$)}  If $U$ takes values in a bounded ball $B_M\subset\mathbb{R}^n$ centered at the origin with radius $M>0$, then
\be\tag{$\text{H}^{B}_{{par}}$}\label{S2:HBpar}
\begin{aligned}
&|B_{\alpha,\beta}(U,x,t)|+ |\nabla B_{\alpha,\beta}(U,x,t)|+|B_{\alpha,\beta,x_\alpha} (U,x,t)|\le C,
\end{aligned}
\ee
$\forall U\in B_M, \,x\in\mathbb{R}^d,\, t>0$, for some constant $C$ possibly depending on $M$. Also, the boundedness conditions~\eqref{S2:HB} and \eqref{S2:HBxt} hold true for higher derivatives of the constitutive functions with respect to $U$.

\begin{remark}
The hypotheses~\eqref{S2:HB}, \eqref{S2:HBxt} and \eqref{S2:HBpar} regarding the boundedness of the constitutive functions with respect to $x$ and $t$ are quite reasonable and they can be validated as part of the existence theory. For instance the presence of $\psi(x,t)\in L^1(\mathbb{R}\times[0,\infty)$ as an upper bound in various functions in~\cite{DH} yields such hypotheses.
\end{remark}

%
\section{The Relative Entropy Identity for Inhomogeneous Systems}\label{sec3}
 We derive the relative entropy inequality for hyperbolic systems~\eqref{Sec3:Ueq} and for hyperbolic-parabolic systems~\eqref{Sec2:Ueqpar}  under the hypotheses set in the previous section.

\subsection{Systems of Balance Laws}\label{sec3.1}

Let $U$ be a weak solution to~\eqref{Sec3:Ueq} that satisfies the entropy inequality~\eqref{Sec3:etaeq} in the distribution sense and $\bar U$ a strong solution to~\eqref{Sec3:barUeq}, hence it satisfies~\eqref{Sec3:etaeq} as an identity. Then we immediately get the inequality
\be\label{Sec3:rel step1}
\del_t(\eta(U,x,t)-\bar \eta)+\del_\alpha(q_\alpha(U,x,t)-\bar q_\alpha) +G\cdot R -\bar G\cdot \bar R \le \eta_t+q_{\alpha,x_\alpha}-\bar \eta_t-\bar q_{\alpha,x_\alpha}\;.
\ee
Also, we can compute
\begin{align}\label{Sec3:rel step2}
\partial_t&\left[\bar G\cdot(A-\bar A)\right]+\partial_\alpha\left[\bar G\cdot (f_\alpha-\bar f_\alpha)\right]=\nonumber\\
=& \nabla\bar G\,\partial_t\bar U\cdot (A-\bar A)+\bar G_t \cdot(A-\bar A)+\bar G\cdot\partial_t (A-\bar A)\nonumber\\
&+\nabla\bar G\,\partial_\alpha \bar U\cdot(f_\alpha-\bar f_\alpha)+\bar G_{x_\alpha} \cdot(f_\alpha-\bar f_\alpha)+\bar G\cdot\partial_\alpha(f_\alpha-\bar f_\alpha)\nonumber\\
=&\nabla\bar G\left[-\left((\nabla \bar A)^{-1}\nabla \bar f_\alpha\partial_\alpha\bar U+ (\nabla \bar A)^{-1} \bar R\right)\cdot (A-\bar A) +\partial_\alpha \bar U\cdot(f_\alpha-\bar f_\alpha)
\right]\nonumber\\
&+\bar G_t \cdot (A-\bar A)+\bar G_{x_\alpha} \cdot (f_\alpha-\bar f_\alpha)-\bar G\cdot(P-\bar P)\nonumber\\
=&\nabla\bar G\, \partial_\alpha \bar U\cdot  f_\alpha(U|\bar U; x,t) 
-\bar R\cdot \nabla\bar G(\nabla \bar A)^{-1} (A-\bar A)\nonumber\\
&+\bar G_t \cdot(A-\bar A)+\bar G_{x_\alpha}\cdot (f_\alpha-\bar f_\alpha)-\bar G\cdot(P-\bar P)\,,
\end{align}
using \eqref{Sec3:Ueq},~\eqref{Sec3:barUeq} and~\eqref{Sec3:relative f} and taking into account hypothesis~{$(\text{H}_1)$}.
Next, we can rewrite the terms 
\begin{align}\label{Sec3:rel step3}
G\cdot R-\bar G\cdot \bar R-\bar R\cdot \nabla\bar G(\nabla& \bar A)^{-1} (A-\bar A)-\bar G \cdot (P-\bar P)
=(G-\bar G) \cdot (R-\bar R) \nonumber\\
&+\bar R \cdot G(U|\bar U;x,t)+\bar G\cdot \left[ A_t-\bar A_t+f_{\alpha,x_\alpha-}\bar f_{\alpha,x_\alpha}\right]\,.
\end{align}

Combining~\eqref{Sec3:rel step1}--\eqref{Sec3:rel step3} with the relative quantities~\eqref{Sec3:relative eta}--\eqref{Sec3:relative f}, we arrive at the relative entropy inequality
\begin{align}\label{Sec3:relentropyhyp}
\del_t(\eta(U|\bar U;x,t)&)+\del_\alpha(q_\alpha(U|\bar U,x,t))+(G-\bar G)\cdot (R -\bar R) \le  -\nabla \bar G\del_\alpha \bar U\cdot  f_\alpha(U|\bar U;x,t)\nonumber\\
&-\bar R \cdot G(U|\bar U;x,t)+ \eta_t+q_{\alpha,x_\alpha}-\bar \eta_t-\bar q_{\alpha,x_\alpha}   \nonumber\\
& -\bar G\cdot(A_t+f_{\alpha,x_\alpha}-\bar A_t-\bar f_{\alpha,x_\alpha}) -\bar G_t \cdot(A-\bar A) -\bar G_{x_\alpha}\cdot (f_\alpha-\bar f_\alpha)\;.
\end{align}

We can reach~\eqref{Sec3:relentropyhyp} under hypotheses~{$(\text{H}_1)$}--\eqref{S2:H2} and we will see in the next section how to prove statements by combining with the other hypotheses.
Certainly, the above expression reduces to (2.20) in~\cite[pp.10]{christoforou2016relative}  (for $P\equiv 0$) or (3.35) in~\cite[pp.24]{christoforou2016relative} (for $P\ne 0$) when there is no explicit dependence of the constitutive functions on $x$ and $t$.

\subsection{Hyperbolic-Parabolic systems}\label{S3.2}
Here, we perform similar calculations for the hyperbolic parabolic systems~\eqref{Sec2:Ueqpar}
with viscosity matrices $B_{\alpha\beta}$, $\alpha,\beta=1,\dots, d$ that depend explicilty on $(x,t)$ and not only on the state $U$. For completeness, we present the computations to reach the relative entropy identity for system~\eqref{Sec2:Ueqpar}. Following the same steps as in the previous subsection, we consider two solution $U$ and $\bar U$ to~\eqref{Sec2:Ueqpar} and for simplicity, we assume that they are both strong solutions. Hence, we get
\begin{align}\label{Sec3.2:rel step1}
\del_t(\eta(U,x,t)-&\bar \eta)+\del_\alpha(q_\alpha(U,x,t)-\bar q_\alpha) +G\cdot R -\bar G\cdot \bar R = \eta_t+q_{\alpha,x_\alpha}-\bar \eta_t-\bar q_{\alpha,x_\alpha}\nonumber\\
&+\e\del_\alpha(G\cdot B_{\alpha\beta}\del_\beta U-\bar G\cdot \bar B_{\alpha\beta}\del_\beta \bar U ) \nonumber\\
& -\e\left[ (\nabla G \del_\alpha U+G_{x_\alpha})\cdot  B_{\alpha\beta}\del_\beta U- (\nabla \bar G \del_\alpha \bar U+\bar G_{x_\alpha})\cdot \bar B_{\alpha\beta}\del_\beta \bar U\right]
\end{align}
and
\begin{align}\label{Sec3.2:rel step2}
\partial_t&\left[\bar G\cdot(A-\bar A)\right]+\partial_\alpha\left[\bar G\cdot (f_\alpha-\bar f_\alpha)\right] +\bar G\cdot(P-\bar P) =\nonumber\\
=&\nabla\bar G\, \partial_\alpha \bar U\cdot  f_\alpha(U|\bar U; x,t) 
-\bar R\cdot \nabla\bar G(\nabla \bar A)^{-1} (A-\bar A)\nonumber\\
&+\bar G_t \cdot(A-\bar A)+\bar G_{x_\alpha}\cdot (f_\alpha-\bar f_\alpha)\nonumber\\
&+\e \nabla\bar G(\nabla \bar A)^{-1}\del_\alpha(\bar B_{\alpha\beta}\del_\beta \bar U)\cdot (A-\bar A)+\e\bar G\cdot\del_\alpha( B_{\alpha\beta}\del_\beta U-\bar B_{\alpha\beta}\del_\beta\bar U)
\end{align}
to arrive at
\begin{align}\label{Sec3.2:relentropypar1} 
\del_t(\eta(U|\bar U&;x,t))+\del_\alpha(q_\alpha(U|\bar U,x,t))+(G-\bar G)\cdot (R -\bar R) =  -\nabla \bar G\del_\alpha \bar U\cdot  f_\alpha(U|\bar U;x,t)\nonumber\\
&-\bar R \cdot G(U|\bar U;x,t) + \eta_t+q_{\alpha,x_\alpha}-\bar \eta_t-\bar q_{\alpha,x_\alpha}  \nonumber\\
& -\bar G\cdot(A_t+f_{\alpha,x_\alpha}-\bar A_t-\bar f_{\alpha,x_\alpha}) -\bar G_t \cdot(A-\bar A) -\bar G_{x_\alpha}\cdot (f_\alpha-\bar f_\alpha)+\e J\,.
\end{align}
Here, the $\e$-terms are
\begin{align}
J:&=\del_\alpha\Big(G\cdot B_{\alpha\beta}\del_\beta U-\bar G\cdot\bar B_{\alpha\beta}\del_\beta\bar U)
-\bar B_{\alpha\beta}\del_\beta\bar U\cdot(\nabla\bar A)^{-T}\nabla\bar G^{T}(A-\bar A)\nonumber\\
&\,\qquad-\bar G\cdot( B_{\alpha\beta}\del_\beta U-\bar B_{\alpha\beta}\del_\beta\bar U)\Big)\nonumber\\
&-(\nabla G \del_\alpha  U -\nabla \bar G \del_\alpha  \bar U)\cdot (B_{\alpha\beta}\del_\beta U- \bar B_{\alpha\beta}\del_\beta \bar U)\nonumber\\
&+ \bar B_{\alpha\beta}\del_\beta \bar U\cdot\Big [-\nabla G \del_\alpha U+\nabla\bar G\del_\alpha\bar U+\nabla\bar G\del_\alpha((\nabla \bar A)^{-1}(A-\bar A))\nonumber\\
&\,\qquad+\nabla^2\bar G:(\del_\alpha\bar U,(\nabla\bar A)^{-1}(A-\bar A))\Big]\nonumber\\
&-(G_{x_\alpha}-\bar G_{x_\alpha})
\cdot\left[ B_{\alpha\beta} ( \del_\beta  U-\del_\beta\bar U)+(B_{\alpha\beta}-\bar B_{\alpha\beta}) \del_\beta \bar U\right]\nonumber\\
&+\bar B_{\alpha\beta}\del_\beta\bar U\cdot\left[-G_{x_\alpha}+\bar G_{x_\alpha}+\nabla\bar G_{x_\alpha}(\nabla\bar A)^{-1}(A-\bar A)\right]\,.
\end{align}
and they can be expressed as $J=\del_\alpha j_\alpha-D+Q$. Here, 
\begin{align}
j_\alpha:
&= (G - \bar G) \cdot ( B_{\alpha \beta} \del_\beta U - \bar B_{\alpha \beta} \del_\beta \bar U ) + 
\bar B_{\alpha \beta}   \del_\beta \bar U \cdot G(U | \bar U;x,t)\nonumber\\
&+\bar B_{\alpha \beta}  \del_\beta\bar U\cdot \nabla \bar G   \phi(U | \bar U;x,t)
\end{align}
for $\alpha=1\dots, d$ and it is part of the fluxes; the next term $D$ is
\be\label{S3.2:D}
D := \sum_{\alpha,\,\beta}\nabla G  \del_\alpha (U - \bar U) \cdot   B_{\alpha \beta} \, \del_{\beta} (U - \bar U)\,,
\ee
that is positive semi-definite and it captures the effect of dissipation by choosing $\xi_\alpha=\del_\alpha(U-\bar U)\in\R^n$; last, the term $Q=\sum_{i=1}^9 Q_i$ is the sum of following terms:
\begin{align}
Q_1 &:= - \del_{x_\alpha} \big ( \bar \nabla G^T \bar B_{\alpha \beta}  \del_\beta \bar U \big ) \cdot \phi(U | \bar U;x,t) \;,
\label{definq1}
\\
Q_2 &:= -  \bar B_{\alpha \beta} \del_\beta \bar U \cdot \big ( \nabla G -  \nabla \bar G  \big ) ( \del_\alpha U- \del_\alpha\bar U)\;,
\label{definq2}
\\
Q_3 &:= - \bar B_{\alpha \beta} \del_\beta \bar U \cdot G_1(U | \bar U;x,t) \del_\alpha\bar U\,,
\label{definq3}
\\ 
Q_4 &:= - \nabla G ( \del_\alpha U - \del_\alpha \bar U) \cdot ( B_{\alpha \beta} - \bar B_{\alpha \beta}  ) \del_\beta\bar  U\,,
\label{definq4}
\\
Q_5 &:= -  ( \nabla G  - \nabla \bar G  ) \del_\alpha \bar U \cdot B_{\alpha \beta} ( \del_\beta U - \del_\beta\bar  U )\,,
\label{definq5}
\\
Q_6 &:=  -  ( \nabla G  - \nabla \bar G  )\del_\alpha \bar U \cdot ( B_{\alpha \beta}  - \bar B_{\alpha \beta}  ) \del_\beta\bar  U\,,
\label{definq6}
\\
Q_7 &:=  - ( G_{x_\alpha}  -\bar G_{x_\alpha}  )\cdot  B_{\alpha \beta} (\del_\beta U-\del_\beta \bar U)  \,,
\label{definq7}
\\
Q_8 &:=  -  ( G_{x_\alpha}  -\bar G_{x_\alpha}  )\cdot ( B_{\alpha \beta}  - \bar B_{\alpha \beta}  ) \del_\beta\bar  U\,,
\label{definq8}
\\
Q_9 &:=  -  \bar B_{\alpha\beta}  \del_\beta\bar  U \cdot G_2(U|\bar U;x,t) \;.
\label{definq9}
\end{align}
%
Above, we set 
\begin{align}
\label{defin1}
\phi (U | \bar U;x,t) &:= (\nabla \bar A)^{-1} (A - \bar A )  - (U - \bar U)\;,
\\
\label{defin2}
G_1(U | \bar U;x,t) &:= \nabla G- \nabla \bar G  - \nabla^2 \bar G  \cdot (\nabla \bar A)^{-1} (A - \bar A )\;,\\
\label{defin3}
G_2(U | \bar U;x,t) &:=G_{x_\alpha} - \bar G_{x_\alpha} - \nabla \bar G_{x_\alpha} \cdot (\nabla  \bar A)^{-1} (A - \bar A )\;,
\end{align}
as the quadratic parts of the expansions of $U$, $\nabla G$ and $G_{x_\alpha}$, respectively. All terms in $Q$ can be viewed as errors, since they are quadratic as $|U - \bar U| \to 0$. Thus, relation~\eqref{Sec3.2:relentropypar1} takes the form
\begin{align}\label{Sec3.2:relentropypar} 
\del_t(\eta(U|\bar U&;x,t))+\del_\alpha(q_\alpha(U|\bar U,x,t)-\e j_\alpha)+(G-\bar G)\cdot (R -\bar R) +\e D \le \nonumber\\
\le & -\nabla \bar G\del_\alpha \bar U\cdot  f_\alpha(U|\bar U;x,t)
-\bar R \cdot G(U|\bar U;x,t) + \eta_t+q_{\alpha,x_\alpha}-\bar \eta_t-\bar q_{\alpha,x_\alpha}  \nonumber\\
& -\bar G\cdot(A_t+f_{\alpha,x_\alpha}-\bar A_t-\bar f_{\alpha,x_\alpha}) -\bar G_t \cdot(A-\bar A) -\bar G_{x_\alpha}\cdot (f_\alpha-\bar f_\alpha)+\e Q\;.
\end{align}
We remark that due to inhomogeneity, the terms $Q_7$, $Q_8$ and $Q_9$ arise in the above calculations and they are not present in the work~\cite{christoforou2016relative}. The interesting point is that these are still expressed in a similar manner as the the other terms $Q_i$, $i\in\{1,\dots,6\}$ preserving the quadratic growth in $|U-\bar U|$.


%
%
%

\section{Theorems}\label{sec4}

In this section, we prove theorems for the hyperbolic inhomogeneous system and the hyperbolic-parabolic one in the next three subsections. 

To prepare the ground for the proofs of the theorems, we establish in the next lemma bounds on the relative entropy that are useful to interpret it as a ``distance formula" and also compare it with other relative quantities that are present in~\eqref{Sec3:relentropyhyp} and \eqref{Sec3.2:relentropypar}. We recall that by $B_M$, we denote a ball in $\mathbb{R}^n$ centered at the origin with radius $M>0$.

\begin{lemma}\label{S4:lemma}
Assume that Hypotheses A $($i.e.~$(\text{H}_1)$--\eqref{S2:H3}, \eqref{S2:Hgr1}--\eqref{S2:Hgr3} and~\eqref{S2:HB}$)$ hold true and let the state $\bar U$ take values in the ball $B_M$. Then there exist $r_1>M$ and $r_2>M$ and positive constants $c_1$, $c_1'$, $c_2$ and $c_2'$ depending only on $M$ such that
\be\label{S4:lemform1}
\eta(U|\bar U;x,t)\ge 
\left\{
\begin{array}{lr}
c_1|A(U,x,t)-\bar A|^2 & \text{if}\quad |U|\le r_1,\,\,
\\
c_2\eta(U,x,t) & \text{if}\quad |U|\ge r_1,\,\,

\end{array}
\right.
\ee
and
\begin{equation}
\label{lemform3}
\eta(U|\bar U;x,t)\ge 
\begin{cases}
c_1'  \big | U - \bar U \big |^2        &  \text{if}\quad \; |U| \le r_2, 
\\
c_2'  \,   \big | U - \bar U \big |^p    &  \text{if}\quad \; |U| \ge r_2,
\end{cases}
\end{equation}
for all $x\in\mathbb{R}^d$ and $t>0$. Moreover, for each $\alpha=1, ..., d$
\begin{align}
&| f_\alpha (U|\bar U;x,t) | \le c_3  \eta ( U|\bar U;x,t )   \quad &\mbox{for $U \in \R^n$, }\label{lemform2}
\end{align}
for all $x\in\mathbb{R}^d$ and $t>0$. 
Assuming further~\eqref{S2:HR1} and ~\eqref{S2:HG}, there exists a constant $c_3$ depending on $M$ such that 
\begin{align}
&|(G(U,x,t)-\bar G)(R(U,x,t)-\bar R)| \le c_3  \eta ( U|\bar U;x,t )   \quad &\mbox{for $U \in \R^n$,} \label{lemform3}\\
&|G(U|\bar U,x,t)|\le c_3  \eta ( U|\bar U;x,t )   \quad& \mbox{for $U \in \R^n$, }\label{lemform4}
\end{align}
for all $x\in\mathbb{R}^d$ and $t>0$. 
\end{lemma}
\begin{proof}
The proof follows in similar lines to~\cite[App. A]{christoforou2016relative}. Here, we present the main steps and skip the details emphasizing how to overcome the explicit dependence on $x$ and $t$. 

First, we recall that
\begin{equation}
\label{workrelen}
\begin{aligned}
\eta ( U | \bar U;x,t ) &\doteq \eta (U,x,t)- \eta (\bar U,x,t )- G(\bar U,x,t) \cdot (A(U,x,t) - A(\bar U,x,t))
\\
&= \tilde{\eta}(V,x,t)  -\tilde{\eta}(\bar V,x,t) - (\nabla_V \tilde{\eta}) (\bar V,x,t)  \; (A(U,x,t)  - A(\bar U,x,t)) \, 
\\
&\doteq \tilde{\eta} \big (V | \bar V;x,t\big ) \, ,
\end{aligned}
\end{equation}
when $V\equiv A(U,x,t)$. By hypothesis~\eqref{S2:H3}, since $\tilde{\eta}$ is uniformly convex on compact subsets of $\R^n$, we get that $\eta ( U | \bar U;x,t )$ is uniformly positive for all $U\ne \bar U$ and all $x$ and $t$. Moreover, by applying also hypothesis $(\text{H}_1)$, we deduce that $\eta ( U | \bar U;x,t )=0$ if and only if $U=\bar U$.
Since $\bar U\in B_M$, by \eqref{S2:Hgr1} and \eqref{S2:HB}, we get
\begin{align}
\eta ( U | \bar U;x,t ) &\ge \eta (U,x,t) - C_1 - C_2 |A(U,x,t) | \, .
\end{align}
Then we proceed as in~\cite[App.A]{christoforou2016relative}  using \eqref{S2:Hgr1}, \eqref{S2:Hgr3}  to select $r_1>M$ such that
$$
\eta ( U | \bar U;x,t )
 \ge\frac{1}{4} \eta(U,x,t) \qquad \qquad   \mbox{for $|U| \ge r_1$, $x\in\mathbb{R}^d$, $t>0$}\;.
$$
In the complement $|U| \le r_1$, we use~\eqref{workrelen}$_2$ and that $\tilde{\eta}(V)$ is uniformly convex in $V$  on compact subsets of $\R^n$. Thus, ~\eqref{S4:lemform1} follows.

Next, we proceed as above selecting $r_2$ possibly even larger than $r_1$ such that
$$
\begin{aligned}
\eta ( U | \bar U;x,t )  &\ge \frac{1}{4}\eta (U,x,t) \ge \frac{\beta_1}{8} |U|^p 
\ge \frac{\beta_1}{16} |U - \bar U |^p \qquad  |U| \ge  r_2 , 
\end{aligned}
$$
using~\eqref{S2:Hgr1} for $|U|\ge r_2$ and $\bar U\in B_M$. Also, by the uniform convexity of $\tilde{\eta}(V)$ in $V=A(U,x,t)$ and the invertibility of $A(\cdot,x,t)$, we get
\begin{equation}
\label{equivnorm2}
 \bar c_1' |U - \bar U |^2 \le \eta ( U | \bar U;x,t ) \le C_1' | U - \bar U |^2    \qquad \qquad   \mbox{for $|U| \le r_2$, } \, 
\end{equation}
since $\bar U \in B_M \subset B_{r_2}$ for some constants $c_1'$, $C_1'$ that correspond to the infimum and supremum of the Hessian $\nabla^2\tilde{\eta}(V,x,t)$ in the domain $|U|\le r_2$ multiplied by a bound on the inverse of $\nabla A$.

To establish \eqref{lemform2}, we estimate
$$
\begin{aligned}
| f_\alpha (U | \bar U;x,t)  | &=  | f_\alpha (U,x,t) - \bar f_\alpha  - \nabla \bar f_\alpha   \nabla \bar A^{-1} (A (U,x,t)- \bar A) |
\\
&\le  | f_\alpha (U,x,t) | + k_1 |A(U,x,t)| + k_2
\end{aligned}
$$
for $\bar U\in B_M$ and some positive constants $k_1$ and $k_2$ that follow from the boundedness condition~\eqref{S2:HB}. Then by~\eqref{S2:Hgr2}--\eqref{S2:Hgr3}, we immediately get
$$
| f_\alpha (U | \bar U;x,t)  |  \le k_3 \eta (U,x,t)  \qquad \qquad   \mbox{for $|U| \ge r_1$} ,
$$
for some positive constant $k_3$ and the radius $r_1$ as selected above. In the complement $|U|<r_1$, we estimate
\begin{align}\label{S4:lemmafinal}
| f_\alpha (U | \bar U;x,t)  | 
&\le | f_\alpha  - \bar f_\alpha - \nabla \bar f_\alpha  (U-\bar U)) |\nonumber
\\
&\qquad +  \left |  \nabla \bar f_\alpha   \nabla \bar A^{-1} \big (A(U,x,t) - \bar A - \nabla \bar A (U - \bar U)  \big ) \right |\nonumber
\\
&\le  k_4 |A(U,x,t) - \bar A|^2  \, ,
\end{align}
for $|U| \le r_1$, where $k_4$ is some positive constant. Then~\eqref{lemform2} follows by combining~\eqref{S4:lemmafinal} with~\eqref{S4:lemform1}. Following the same steps, ~\eqref{lemform3} and ~\eqref{lemform4} hold true assuming ~\eqref{S2:HR1} and ~\eqref{S2:HG}, respectively.
%
\end{proof}

\subsection{Weak-strong uniqueness in the class of dissipative measure valued solutions}\label{S4.1}

The aim is to prove a weak-strong uniqueness result comparing two solutions, a weak solution $U$ and a strong solution $\bar U$, to~\eqref{Sec3:Ueq} with respect to their initial data. We do this in the context of \emph{dissipative measure-valued} solutions $U$ that is a more general class of entropy weak solutions. To arrive at the definition of a \emph{dissipative measure-valued} solution to~\eqref{Sec3:Ueq}, we exploit briefly how they arise from a sequence of approximate solutions  $U^\eps$ to
\begin{equation}
\label{hypcleps}
\begin{aligned}
\del_t (A(U^\eps,x,t))+\del_\alpha(f_\alpha(U^\eps,x,t))+P(U^\eps,x,t)= \mathcal{R}_1^\eps,
\end{aligned}
\end{equation}
satisfying an entropy inequality
\begin{equation}
\label{entropyhypeps}
\begin{aligned}
\del_t(\eta(U^\eps,x,t))+\del_\alpha q_\alpha(U^\eps,x,t) +Z(U^\eps,x,t)
\le \mathcal{R}_2^\eps,
\end{aligned}
\end{equation}
with both $\mathcal{R}_1^\eps\to0$ and $\mathcal{R}_2^\eps\to0$ in distributions as $\eps\to0+$.

To simplify the analysis, throughout this subsection, we consider the spatially periodic case with domain 
$\To^d =(\R/2\pi\Z)^d$ 
and denote by $Q_T:= \To^d \times[0,T)$ for $T\in(0,\infty)$ the domain of the solution and $\overline{Q}_T:= \To^d \times[0,T]$.

Setting $V = A(U,x,t)$ and using the invertibility of $A(\cdot,x,t)$ for each $(x,t)$, we write  $\eta (U,x,t) = \tilde{\eta} (A(U,x,t),x,t)$. Having that $\tilde{\eta}(V,x,t)$ is convex and positive, we assume further that there are positive constants $\beta_i'$, $i=1,2,3$ such that
 \begin{equation}\label{S2.2 A4}
\frac{1}{\beta_1'}( |A(U,x,t)|^q + 1) - \beta_3' \le \tilde{\eta}(A(U,x,t))\le \beta_2'( |A(U,x,t)|^q + 1) ,
\tag{H$\,_{gr}^*$}
\end{equation}
for all $(x,t)\in Q_T$ and some $q>1$.

Now, let us assume that 
$\{U^\eps\}$ is a sequence of Lebesgue measurable functions with a convergent subsequence (again called $U^\eps$) that is associated with the Young measure  $\boldsymbol{\nu}_{(x,t)}$, which is a weak$^*$ measurable family of Radon probability measures (cf.~\cite{MR725524, MR1036070, ab97, dm87}). Then, for all continuous functions $f(\cdot,x,t)$ such that $\displaystyle\lim_{|\lambda|\to\infty}\dfrac{f(\lambda,x,t)}{1+|\lambda|^p}=0$ and for all $(x,t)$, we get in the limit
\begin{equation}
f(U^\eps,x,t)\rightharpoonup\langle{\boldsymbol{\nu}}_{x,t},f(\lambda,x,t)\rangle
\end{equation}
using the Young measure $\boldsymbol{\nu}=({\boldsymbol{\nu}}_{(x,t)})_{(x,t)\in\overline{Q_T}}$ associated to the family $\{U^\eps\}$.

If the associated Young measure to the sequence $\{V^\eps\}$ is denoted by ${\boldsymbol{N}}=({\boldsymbol{N}}_{(x,t)})_{(x,t)\in\overline{Q_T}}$, we have
\begin{equation}\label{nu-N}
\langle{\boldsymbol{\nu}}_{x,t},g(A(\lambda,x,t),x,t)\rangle=\langle{\boldsymbol{\nu}}_{x,t},f(\lambda,x,t)\rangle=\langle{\boldsymbol{N}}_{x,t},g(\rho,x,t)\rangle
\end{equation}
whenever $f(\lambda,x,t)=g(A(\lambda,x,t),x,t)$.

As indicated in~\cite{christoforou2016relative}, the relation $\eta(\cdot,x,t)$ via $\eta(U,x,t)=\tilde{\eta}(A(U,x,t),x,t)$ and hypothesis~\eqref{S2:H3} imply the convexity of $\tilde{\eta}(\cdot,x,t)$ (recall~\eqref{appform2}) 
and the oscillations and concentrations can be represented via
\begin{equation}
\eta(U^\eps,x,t) dxdt\rightharpoonup\langle{\boldsymbol{\nu}}_{x,t},\eta(\cdot,x,t)\rangle dxdt+\boldsymbol{\gamma}(dxdt)
\end{equation}
as $\eps\to0+$, 
with the oscillations obtained via the limiting process
\begin{equation}
\langle{\boldsymbol{\nu}}_{x,t},\eta(\lambda,x,t)\rangle\doteq \lim_{R\to\infty}\langle{\boldsymbol{N}}_{x,t},\tilde{\eta}(A(\lambda,x,t),x,t)\cdot 1_{\tilde{\eta}(A,x,t)<R}+R\cdot 1_{\tilde{\eta}(A,x,t)\ge R}\rangle
\end{equation}
while the \emph{concentration measure} $\boldsymbol{\gamma}$ expressed as
\begin{equation}
\boldsymbol{\gamma}=\text{wk}^*-\lim_{\eps\to0+}(\eta(U^\eps,x,t)- \langle{\boldsymbol{\nu}}_{x,t},\eta(\cdot,x,t)\rangle  )\in\mathcal{M}^+(\overline{Q_T}).
\end{equation}

Now we state the definition of \emph{dissipative measure-valued solutions}, that satisfy an averaged and integrated form of the entropy inequality with concentration effects in the $L^p$ framework $p<\infty$.

\begin{definition}\label{defdissi}
A dissipative measure valued solution $(U,\boldsymbol{\nu},\boldsymbol{\gamma})$ with concentration to~\eqref{Sec3:Ueq} consists of $U\in L^\infty(L^p)$, a Young measure $\boldsymbol{\nu}=({\boldsymbol{\nu}}_{x,t})_{\{(x,t)\in \bar{Q}_T\}}$ and a non-negative Radon measure $\boldsymbol{\gamma}\in\mathcal{M}^+(Q_T)$ such that
$U (x,t)  = \langle {\boldsymbol{\nu}}_{(x,t)} , \lambda \rangle$ and
\begin{align}\label{dfmv1}
\iint&\langle {\boldsymbol{\nu}}_{x,t}, A_i(\lambda,x,t)\rangle \del_t\varphi_{i}\,dx\,dt+\iint \langle {\boldsymbol{\nu}}_{x,t}, f_{i,\alpha}(\lambda,x,t)\rangle \del_\alpha\varphi_{i}dx\,dt\nonumber\\
&\qquad- \iint \langle {\boldsymbol{\nu}}_{x,t}, P_i(\lambda,x,t)\rangle \varphi_{i}(x,t)dx\,dt
+\int \langle\boldsymbol{\nu}_0,A_i(\lambda,x,0)\rangle\varphi_i(x,0)dx=0
\end{align}
$i=1,\dots,n,$ for any $\varphi\in C^1_c(Q\times[0,T))$ and
\begin{align}\label{dfmv2}
\iint\frac{d\xi}{dt}& \left[\langle {\boldsymbol{\nu}}_{x,t}, \eta(\lambda,x,t)\rangle dxdt+\boldsymbol{\gamma}(dxdt)    \right]+\int \xi(0)\left[\langle {\boldsymbol{\nu}_0}_{x}, \eta(\lambda,x,0)\rangle dx+\boldsymbol{\gamma}_0(dx)  \right]\nonumber\\
&-\iint\xi(t) \langle {\boldsymbol{\nu}}_{x,t}, Z(\lambda,x,t)\rangle dx\,dt
\ge 0,
\end{align}
 for all $\xi=\xi(t)\in C^1_c([0,T))$ with $\xi\ge 0$.
\end{definition}

Recall $Z:=G\cdot R-\eta_t-q_{\alpha,x_\alpha}$ appears in~\eqref{Sec3:etaeq}. Also, we note that a generalization of the above definition can be given following the work in~\cite{GOA} taking into account that concentration measures associated to all constitutive functions $A$ $f_\alpha$, $P$ may be present. 

Now, we prove uniqueness in the $L^p$ framework for $1<p<\infty$ within a class of dissipative measure-valued solutions as defined in Definition~\ref{defdissi} as long there exists a strong solution $\bar U$ having same initial data and no initial concentration measure $\boldsymbol{\gamma}_0=0$. This result is an extension of previous ones with no inhomogeneity for $L^\infty$ or $L^p$, (cf.~\cite{bds11, GSW2015, FKMT, christoforou2016relative}).
%
%
\begin{theorem}
\label{thmweakstrong} 
Assume that Hypotheses A and B hold true and also \eqref{S2.2 A4} is satisfied, while the entropy $\eta$ is nonnegative. Let $(U,\boldsymbol{\nu}, \boldsymbol{\gamma})$ be a dissipative measure--valued solution,  $U = \langle \boldsymbol{\nu}_{x,t} , \lambda \rangle$,  and  $\bar U\in W^{1,\infty}(\overline{Q_T})$ a strong solution to~\eqref{Sec3:Ueq}, then, if the  initial data of $U$
satisfy $\boldsymbol{\gamma}_0=0$ and ${\boldsymbol{\nu}_0}_x=\delta_{\bar U_0}(x)$, it holds $\boldsymbol{\nu}=\delta_{\bar{U}}$ and $U=\bar{U}$ almost everywhere on $Q_T$.
\end{theorem}
\begin{proof}
First we define the averaged quantities of the relative entropy
\begin{equation}\label{dfH}
\mathcal{H}(\boldsymbol{\nu},U,\bar{U};x,t)\doteq\langle \boldsymbol{\nu},\eta\rangle-\bar\eta-\bar G\cdot \big ( \langle\boldsymbol{\nu},A\rangle-\bar A \big )
\end{equation}
that satisfies
\begin{equation}
\mathcal{H}(\boldsymbol{\nu},U,\bar{U};x,t)=\int \eta(\lambda|\bar{U};x,t)d\boldsymbol{\nu}(\lambda)
\end{equation}
using definition~\eqref{Sec3:relative eta}.
We proceed  in the spirit of Subsection~\ref{sec3.1} but for the averaged relations. Hence, we first get 
\begin{align}\label{Sec4.1:rel step2}
\del_t \bar G \cdot (\langle\boldsymbol{\nu},A\rangle-\bar A)+\del_{\alpha} \bar G\cdot&(\langle\boldsymbol{\nu},f_\alpha\rangle-\bar f_\alpha)=\nabla\bar G\, \partial_\alpha \bar U\cdot  \langle\boldsymbol{\nu}, f_\alpha(\lambda|\bar U,x,\tau)\rangle
\nonumber\\
&
-\bar R\cdot \nabla\bar G\nabla \bar A^{-1} (\langle\boldsymbol{\nu},A\rangle-\bar A)\nonumber\\
&+\bar G_t \cdot(\langle\boldsymbol{\nu},A\rangle-\bar A)+\bar G_{x_\alpha}\cdot (\langle\boldsymbol{\nu},f_\alpha\rangle -\bar f_\alpha)
\end{align}
for $\alpha=1,\dots,d$ that corresponds to~\eqref{Sec3:rel step2}. Since $\bar{U}\in W^{1,\infty}(\overline{Q_T})$ is a strong solution of~\eqref{Sec3:Ueq}, then it verifies the strong versions of~\eqref{dfmv1}--\eqref{dfmv2}, i.e. the entropy relation~\eqref{dfmv2} as an identity.

Now choosing $\varphi(x,\tau)\doteq\xi(\tau) G(\bar U(x,\tau),x,\tau)$ in Definition~\ref{defdissi} and applying~\eqref{dfmv1} for both $U= \langle \boldsymbol{\nu}_{x,t} , \lambda \rangle$ and $\bar U$ in combination with~\eqref{Sec4.1:rel step2}, we arrive at
\begin{align}\label{intmvdiffhyp}
\iint\frac{d\xi}{d\tau} \bar G \cdot(&\langle\boldsymbol{\nu}, A\rangle-\bar A))+\xi(\tau)\Big[ \nabla \bar G\del_\alpha \bar{U} \cdot \langle\boldsymbol{\nu}, f_\alpha(\lambda|\bar U,x,\tau)\rangle
\nonumber
\\
&-\bar R\cdot \nabla\bar G\nabla \bar A^{-1} (\langle\boldsymbol{\nu},A\rangle-\bar A)
+\bar G_t \cdot(\langle\boldsymbol{\nu},A\rangle-\bar A)+\bar G_{x_\alpha}\cdot (\langle\boldsymbol{\nu},f_\alpha\rangle -\bar f_\alpha)\Big]dxd\tau\nonumber\\
&-\iint \xi(\tau) \bar G\cdot(\langle\boldsymbol{\nu},P\rangle-\bar P) dxd\tau+\int \xi(0)\bar G_0 (\langle\boldsymbol{\nu}_0,A\rangle-\bar A_0)dx=0\;.
\end{align}
%
Then we continue in the usual way to come up with the relative entropy quantity. In other words, we subtract from~\eqref{dfmv2} the entropy identity satisfied by the strong solution $\bar U$ and~\eqref{intmvdiffhyp}
to get
\begin{equation}\label{Hintineq}
\begin{aligned}
\iint\frac{d\xi}{d\tau} \, &\mathcal{H}(\boldsymbol{\nu},U,\bar{U};x,t) dxd\tau  + \iint \frac{d\xi}{d\tau} \gamma (dx d\tau) 
\\
&- \iint \xi(\tau) \Big[ \nabla \bar G\del_\alpha\bar{U} \cdot \langle\boldsymbol{\nu}, f_\alpha(\lambda|\bar U,x,\tau)\rangle
-\bar R\cdot \nabla\bar G\nabla \bar A^{-1} (\langle\boldsymbol{\nu},A\rangle-\bar A)\Big]dxd\tau 
\\
&
-\iint \xi(\tau)\Big[\bar G_t \cdot(\langle\boldsymbol{\nu},A\rangle-\bar A)+\bar G_{x_\alpha}\cdot (\langle\boldsymbol{\nu},f_\alpha\rangle -\bar f_\alpha)\Big]dxd\tau\\
&+\iint \xi(\tau) \bar G\cdot(\langle\boldsymbol{\nu},P\rangle-\bar P) dxd\tau-\iint \xi(\tau) (\langle\boldsymbol{\nu},Z\rangle-\bar Z) dxd\tau\\
&+\int\xi(0)\left[(\langle\boldsymbol{\nu}_0,\eta\rangle-\bar\eta_0-\bar G(0)\cdot (\langle\boldsymbol{\nu}_0,A\rangle-\bar A_0))dx+\boldsymbol{\gamma}_0(dx)\right] \ge 0\,,
\end{aligned}
\end{equation}
for any $\xi \in C_c^1([0,T))$ with $\xi \ge 0$. This inequality can be rewritten as
\begin{equation}\label{S4.2:Hintineq}
\begin{aligned}
\iint\frac{d\xi}{d\tau} \, &\mathcal{H}(\boldsymbol{\nu},U,\bar{U};x,t) dxd\tau  + \iint \frac{d\xi}{d\tau} \gamma (dx d\tau) 
-\iint\xi(\tau) (\langle\boldsymbol{\nu},G\rangle-\bar G )\cdot (\langle\boldsymbol{\nu},R\rangle- \bar R)dxd\tau\\
&\ge \iint \xi(\tau) \Big[ \nabla \bar G\del_\alpha\bar{U} \cdot \langle\boldsymbol{\nu}, f_\alpha(\lambda|\bar U,x,\tau)\rangle
+\bar R\cdot \langle\boldsymbol{\nu}, G(\lambda|\bar U,x,\tau)\rangle
\Big]dxd\tau 
\\
&
+\iint \xi(\tau)\Big[\bar G \cdot(\langle\boldsymbol{\nu},A_t\rangle-\bar A_t+ \langle\boldsymbol{\nu},f_{\alpha,x_\alpha}\rangle -\bar f_{\alpha,x_\alpha})\Big]dxd\tau\\
&
+\iint \xi(\tau)\Big[\bar G_t \cdot(\langle\boldsymbol{\nu},A\rangle-\bar A)+\bar G_{x_\alpha}\cdot (\langle\boldsymbol{\nu},f_\alpha\rangle -\bar f_\alpha)\Big]dxd\tau\\
&-\iint \xi(\tau) \Big[\langle\boldsymbol{\nu},\eta_t\rangle -\bar\eta_t+ \langle\boldsymbol{\nu},q_{\alpha,x_\alpha}\rangle-\bar q_{\alpha,x_\alpha}\Big] dxd\tau\\
&-\int\xi(0)\left[(\langle\boldsymbol{\nu}_0,\eta\rangle-\bar\eta_0-\bar G(0)\cdot (\langle\boldsymbol{\nu}_0,A\rangle-\bar A_0))dx+\boldsymbol{\gamma}_0(dx)\right] 
\end{aligned}
\end{equation}
and it corresponds to the weak version of the relative entropy inequality~\eqref{Sec3:relentropyhyp}.

Next, let $K\subset\mathbb{R}^n$ be a compact set containing the values of the strong solution $\bar U(x,t)$ for $(x,t)\in Q_T$. Then $K$ is contained in a ball $B_M\subset\mathbb{R}^n$ for some radius $M$.

We apply \eqref{S4.2:Hintineq} to a sequence of smooth, monotone nonincreasing functions $\xi_n \ge 0$ that
approximate the Lipschitz function
\begin{equation}\label{S2.2.2xi}
\xi(\tau)\doteq\left\{
\begin{array}{lll}
1 & \text{if} & 0\le \tau<t\\
\frac{t-\tau}{\eps}+1 & \text{if} & t\le \tau<t+\eps\\
0 & \text{if} & \tau\ge t+\eps
\end{array}
\right. \; .
\end{equation}
Passing in~\eqref{S4.2:Hintineq} first to the limit $n \to \infty$ and then to $\eps\to0+$, we arrive at
\begin{equation}\label{Hintineq2}
\begin{aligned}
\int \mathcal{H}(\boldsymbol{\nu},U,&\bar{U};x,t)\,dx\le C\int_0^t \int\max_{\alpha}|\langle\boldsymbol{\nu}, f_\alpha(\lambda|\bar U,x,\tau)\rangle|+|\langle\boldsymbol{\nu}, G(\lambda|\bar U,x,\tau)\rangle|dx\,d\tau\\
&+C\int_0^t \int| (\langle\boldsymbol{\nu},G\rangle-\bar G )\cdot (\langle\boldsymbol{\nu},R\rangle- \bar R)| dx\,d\tau \\
&
+C\int_0^t\int\Big|\langle\boldsymbol{\nu},A_t\rangle-\bar A_t+ \langle\boldsymbol{\nu},f_{\alpha,x_\alpha}\rangle -\bar f_{\alpha,x_\alpha})\Big|dxd\tau\\
&
+C\int_0^t\int \Big[|\langle\boldsymbol{\nu},A\rangle-\bar A|+|\langle\boldsymbol{\nu},f_\alpha\rangle -\bar f_\alpha|\Big]dxd\tau\\
&+C\int_0^t\int  \Big|\langle\boldsymbol{\nu},\eta_t\rangle -\bar\eta_t+ \langle\boldsymbol{\nu},q_{\alpha,x_\alpha}\rangle-\bar q_{\alpha,x_\alpha}\Big| dxd\tau\\
&+\int \left[(\langle\boldsymbol{\nu}_0,\eta\rangle-\bar\eta_0-\bar G(0)\cdot (\langle\boldsymbol{\nu}_0,A\rangle-\bar A_0))dx+\boldsymbol{\gamma}_0(dx)\right]\\
\end{aligned}
\end{equation}
for $t\in(0,T)$ employing hypotheses ~\eqref{S2:HB} and \eqref{S2:HBxt}. We observe that $C$ is a positive constant depending possibly on $(M,|\nabla\bar U|, T)$ through the bounds in~\eqref{S2:HB} and \eqref{S2:HBxt}. We also have used that
$\gamma \ge 0$.

Now, by hypotheses~\eqref{S2:Hgr2}, \eqref{S2:HR1} and ~\eqref{S2:HG}, we can apply
Lemma \ref{S4:lemma}. 
By~\eqref{lemform2}--\eqref{lemform4}, we have
\begin{equation}
\begin{aligned}
|  \langle \boldsymbol{\nu} , f_\alpha (\lambda | \bar U;x,t ) \rangle|
 &\le  c_3 \langle \boldsymbol{\nu} , \eta (\lambda | \bar U;x,t ) \rangle 
 = c_3 \mathcal{H} ( \boldsymbol{\nu},U,\bar{U};x,t)\;,
\end{aligned}
\end{equation}
\begin{equation}
\begin{aligned}
|\langle\boldsymbol{\nu}, G(\lambda|\bar U,x,\tau)\rangle|
 &\le  c_3 \langle \boldsymbol{\nu} , \eta (\lambda | \bar U;x,t ) \rangle 
 = c_3 \mathcal{H} ( \boldsymbol{\nu},U,\bar{U};x,t)\;,
\end{aligned}
\end{equation}
\begin{equation}
\begin{aligned}
| (\langle\boldsymbol{\nu},G\rangle-\bar G )\cdot (\langle\boldsymbol{\nu},R\rangle- \bar R)|
 &\le  c_3 \langle \boldsymbol{\nu} , \eta (\lambda | \bar U;x,t ) \rangle 
 = c_3 \mathcal{H} ( \boldsymbol{\nu},U,\bar{U};x,t)\;.
\end{aligned}
\end{equation}
Let us clarify that under Hypothesis~\eqref{S2:HR2}, the term $| (\langle\boldsymbol{\nu},G\rangle-\bar G )\cdot (\langle\boldsymbol{\nu},R\rangle- \bar R)|$ in~\eqref{Hintineq2} can be omitted. 
Combining~\eqref{S2:Hxt}--\eqref{S2:Hxt3} with~\eqref{S2:Hgr2}--\eqref{S2:Hgr3} and~\eqref{S2:Hgr4}--\eqref{S2:Hgr5}, we can also get
\begin{equation}
\begin{aligned}
\Big|\langle\boldsymbol{\nu},A_t\rangle-\bar A_t+ \langle\boldsymbol{\nu},f_{\alpha,x_\alpha}\rangle -\bar f_{\alpha,x_\alpha})\Big| &\le  c_3 \langle \boldsymbol{\nu} , \eta (\lambda | \bar U;x,t ) \rangle 
 = c_3 \mathcal{H} ( \boldsymbol{\nu},U,\bar{U};x,t)\;.
\end{aligned}
\end{equation}
\begin{equation}
\begin{aligned}
 \Big[|\langle\boldsymbol{\nu},A\rangle-\bar A|+|\langle\boldsymbol{\nu},f_\alpha\rangle -\bar f_\alpha|\Big]
 &\le  c_3 \langle \boldsymbol{\nu} , \eta (\lambda | \bar U;x,t ) \rangle 
 = c_3 \mathcal{H} ( \boldsymbol{\nu},U,\bar{U};x,t)\;.
\end{aligned}
\end{equation}
\begin{equation}
\begin{aligned} \Big|\langle\boldsymbol{\nu},\eta_t\rangle -\bar\eta_t+ \langle\boldsymbol{\nu},q_{\alpha,x_\alpha}\rangle-\bar q_{\alpha,x_\alpha}\Big|  &\le  c_3 \langle \boldsymbol{\nu} , \eta (\lambda | \bar U;x,t ) \rangle 
 = c_3 \mathcal{H} ( \boldsymbol{\nu},U,\bar{U};x,t)\;,
\end{aligned}
\end{equation}
for some positive constant that we call again $c_3$. Indeed these estimates follow immediately using that $\bar U\in B_M$, bound~\eqref{S4:lemform1} and following the ideas in the proof of Lemma~\ref{S4:lemma}.
Taking into account that there are no concentrations at the initial data, i.e. $\boldsymbol{\gamma}_0=0$,
estimate \eqref{Hintineq2} reduces to
\begin{equation}
\int \mathcal{H}(\boldsymbol{\nu},U,\bar{U};x,t)\,dx \le C'\int_0^t\int \mathcal{H}(\boldsymbol{\nu},U,\bar{U};x,\tau)\,+\iint\eta(\lambda|\bar U_0;x,0) d\boldsymbol{\nu}_0(\lambda)\,dx\,.
\end{equation}
Applying Gronwall's inequality, we obtain
\begin{equation}
\int\mathcal{H}(\boldsymbol{\nu},U,\bar{U};x,t)\,dx\le \left[ \iint \eta(\lambda|\bar U_0;x,0)d{\boldsymbol{\nu}}_0(\lambda)dx\right]\, e^{C' t}\;.
\end{equation}
The proof is complete.
\end{proof}

The uniqueness result follows immediately as a corollary in the class of entropy weak solutions. Now, a uniqueness result in the framework of  $L^\infty$ uniform bounds can also be achieved in the class of dissipative measure-valued solutions since in that setting no concentrations are present, i.e. $\boldsymbol{\gamma}=0$ and $\boldsymbol{\gamma}_0=0$ in the definition of the solution. The reader can combine steps from the above proof in conjunction with ~\cite[Theorem 2.2]{dst12} and \cite{bds11}. For completeness, let us state the result.

\begin{theorem}Assume that Hypotheses A and B hold true together with \eqref{S2.2 A4} and that the entropy $\eta$ is nonnegative. 
Let $\bar U\in W^{1,\infty}(Q_T)$ be a strong solution and let $(U,\boldsymbol{\nu})$ be  a dissipative measure valued solution to~\eqref{Sec3:Ueq} respectively. Assume that $\bar U$ and $U$ take value in a compact set $ K \subset\R^n$ for $(x,t)\in Q_T$ and that $\nu$ is also supported in $K$. Then there exist constants $c_1>0$ and $c_2>0$ such that
\begin{equation}
\iint|\lambda-\bar U(t)|^2d\boldsymbol{\nu}(\lambda)dx\le c_1 e^{c_2 t} \int|U_0-\bar U_0|^2dx ,
\end{equation}
for $t\in[0,T]$. Moreover, if the initial data agree $U_0=\bar U_0$, then $\boldsymbol{\nu}=\delta_{\bar U}$ and the dissipative measure valued solution is a strong solution, i.e. $U=\bar U$ almost everywhere.
\end{theorem}

%

\subsection{Stability of Solutions to Hyperbolic-Parabolic Systems of Conservation Laws with Inhomogeneity}\label{S4.2}
\,

In this subsection, we prove stability of viscous solutions to hyperbolic-parabolic system~\eqref{Sec2:Ueqpar} with respect to initial data having the additional hypotheses~\eqref{S2:HP1} and \eqref{S2:HBpar} from the set of Hypotheses C. Let us note that the growth conditions ~\eqref{S2:Hgr1}--\eqref{S2:Hgr3}, \eqref{S2:Hgr4}--\eqref{S2:Hgr5} are not assumed here.
 In this setting, we work with solutions, whose domain is $\R^d\times [0,T]$ and assume that they decay to zero as $|x| \to \infty$. Such a property is usually recovered by the existence theory to such systems and thus, it makes sense to consider it. Let us note that one could also work with periodic solutions on $\bar Q_T=\To^d\times[0,T]$.

The theorem on the $L^2$-stability of viscous solutions with respect to initial data is:
\begin{theorem}
\label{thmstability}
Fix $\e\in(0,\e_0)$ and suppose that $U(x,t)$, $\bar U(x,t)$ are smooth solutions of~\eqref{Sec2:Ueqpar}  defined on $\R^d\times[0,T]$  such that $U$, $\del_\alpha U$ and $\bar U$, $\del_\alpha \bar U$ decay 
to zero sufficiently fast as $|x|\to\infty$ with both having smooth initial data $U_0$, $\bar U_0$ in $ (L^\infty \cap L^2 )(\R^d)$.
Assume that  hypotheses $(\text{H}_1)$--\eqref{S2:H3}, \eqref{S2:HB}, \eqref{S2:Hxt}--\eqref{S2:Hxt3}, \eqref{S2:HBxt} hold true together with hypotheses~\eqref{S2:HP1} and~\eqref{S2:HBpar}. If both $U$ and $\bar U$ take values in a ball $B_M\subset\R^n$ of radius $M>0$, then there exists a constant 
$C>0$ independent of $\eps$,  such that
\begin{equation}\label{S2.3.1stability}
\|U(t)-\bar U(t)\|_{L^2(\R^d)}\le C  \left(\|U_0-\bar U_0\|_{L^2(\R^d)} \right)\;,
\end{equation}
for $t\in[0,T]$.
\end{theorem}

\begin{proof}
By integrating~\eqref{Sec3.2:relentropypar} over $\R^d\times[0,T]$  and using~\eqref{S3.2:D} and hypothesis~\eqref{S2:HP1}, we obtain
\begin{align}\label{Sec4.2:relentropypar} 
\int_{\R^d}&\eta(U(t)|\bar U(t);x,t)\,dx+\eps\lambda_1\int_0^t\int_{\R^d}|\nabla U(s)-\nabla\bar U(s)|^2\,dx\,ds\nonumber\\
\le & \int_{\R^d}\eta(U_0|\bar U_0;x,0)\,dx-\int_0^t\int_{\R^d}\nabla \bar G\del_\alpha \bar U\cdot  f_\alpha(U|\bar U;x,s)+\bar R \cdot G(U|\bar U;x,s) dxds\nonumber\\
&-\int_0^t\int_{\R^d}(G-\bar G)\cdot (R -\bar R) dxds +\int_0^t\int_{\R^d} \eta_t-\bar \eta_t+q_{\alpha,x_\alpha}-\bar q_{\alpha,x_\alpha}dxds  \nonumber\\
& -\int_0^t\int_{\R^d}\bar G\cdot(A_t-\bar A_t+f_{\alpha,x_\alpha}-\bar f_{\alpha,x_\alpha})dxds\nonumber\\
& -\int_0^t\int_{\R^d}\bar G_t \cdot(A-\bar A) +\bar G_{x_\alpha}\cdot (f_\alpha-\bar f_\alpha)dxds+\e\int_0^t\int_{\R^d} Q(x,s)dxds\;.
\end{align}

Since both $U$ and $\bar U$ take values in $B_M$, the terms in~\eqref{Sec3:relative G},~\eqref{Sec3:relative f},~\eqref{defin1}--\eqref{defin3} are all quadratic in $|U-\bar U|$. Indeed this follows by employing \eqref{S2:HBpar}. Hence, using also the boundedness conditions~\eqref{S2:HB}, ~\eqref{S2:HBxt}, this yields a constant $C_1>0$ such that
\begin{equation}\label{S4.2bd1}
\left|\int_{\R^d} \nabla \bar G\del_\alpha \bar U\cdot  f_\alpha(U|\bar U;x,s)+\bar R \cdot G(U|\bar U;x,s) \,dx  \right|
\le C_1 \|U(s)-\bar U(s)\|^2,
\end{equation}
\be
\left| \int_{\R^d}(G-\bar G)\cdot (R -\bar R) dxds \right|\le C_1 \|U(s)-\bar U(s)\|^2,
\ee
\begin{equation}
\left| \int_{\R^d}  Q_1(s)+Q_3(s)+Q_6(s)+Q_8+Q_9\, dx         \right|\le C_1\|U(s)-\bar U(s)\|^2,
\end{equation}
\begin{equation}\label{S2.3.1Q5}
\left| \int_{\R^d} Q_2(s)+Q_4(s)+Q_5(s)+Q_7 dx         \right|\le C_1\|U(s)-\bar U(s)\|_{L^2(\R^d)}^2+\lambda_1\| \nabla U(s)-\nabla\bar U(s)\|^2\,,
\end{equation}
for $s\in[0,t]$. Here, $C_1$ is a universal constant depending on the radius $M$, $\lambda_1$, the derivatives of $\bar U$ and on the uniform bounds assumed in~\eqref{S2:HB}, ~\eqref{S2:HBxt} and~\eqref{S2:HBpar} and $\|\cdot\|$ is the norm in ${L^2(\R^d)}$.
Employing~\eqref{S2:Hxt}--\eqref{S2:Hxt3} and bounds~\eqref{S4.2bd1}--\eqref{S2.3.1Q5} in~\eqref{Sec4.2:relentropypar}, we get
\begin{align}\label{Sec4.2:relentropypar2} 
\int_{\R^d}\eta(U(t)|\bar U(t);x,t)\,dx\le&  \int_{\R^d}\eta(U_0|\bar U_0;x,0)\,dx+2C(1+\e)\int_0^t\|U(s)-\bar U(s)\|^2ds\,.
\end{align}
Now, hypotheses (H$_1$) and~\eqref{S2:H3} (with the constant $\mu$ in \eqref{S2:H3} possibly depending on $M$ since $U,\bar U\in B_M$) allow us to write the relation
\begin{equation}\label{S2.3.1L2}
\frac{1}{C'} \|U(t)-\bar U(t)\|^2  \le \int_{\R^d}\eta(U(t)|\bar U(t);x,t)\,dx \le C'\|U(t)-\bar U(t)\|^2
\end{equation}
for some constant $C'>0$ and this expresses the equivalence of the relative entropy with the $L^2$ norm. Hence, combining~\eqref{Sec4.2:relentropypar2} and \eqref{S2.3.1L2} with Gronwall's lemma for $0<\eps<\eps_0$, we arrive at
\begin{equation}
\|U(t)-\bar U(t)\|_{L^2(\R^d)} ^2\le C_3 e^{C_2(1+\e) t}\left(\|U_0-\bar U_0\|_{L^2(\R^d)}^2\right)
\end{equation}
for $0<\eps<\eps_0$ and $C_2$, $C_3$ positive constants independent of $\e$.  Then~\eqref{S2.3.1stability} follows immediately and the proof is complete. Note that $C$ in~\eqref{S2.3.1stability} depends on $(T,\lambda_1,\mu,M,\nabla\bar U,\e_0)$. 
\end{proof}
%
%
%
%

\subsection{Convergence of viscous approximations to the hyperbolic limit}\label{S4.3}

Now, we consider a sequence of smooth solutions $\{U^\e\}$ to the hyperbolic-parabolic system~\eqref{Sec2:Ueqpar} with domain $\mathbb{R}^d \times [0, T)$ and satisfying the entropy condition~\eqref{hypparaen}. Assume that $\bar U$ is a smooth solution to the hyperbolic system~\eqref{Sec3:Ueq} defined on $ \mathbb{R}^d \times [0, T^*)$ with $T^*\le\infty$ satisfying~\eqref{Sec3:etaeq}. The aim is to prove the convergence $U^\e\to \bar U$ on $\mathbb{R}^d \times [0,T]$, for $T < T^*$ as the viscosity $\e$ tends to zero. It should be noted that here we assume hypothesis~\eqref{S2:HP2} instead of~\eqref{S2:HP1} because after performing the relative entropy calculation (see below) the term in the left-hand side of~\eqref{S2:HP2} arises and can control remaining terms in the relative entropy identity. This condition was motivated by the analysis in Dafermos~\cite[Ch IV]{MR3468916} and it is related to Kawashima condition~\cite{Kawashima84}. See also Serre~\cite{S}. Again, here ~\eqref{S2:Hgr1}--\eqref{S2:Hgr3}, \eqref{S2:Hgr4}--\eqref{S2:Hgr5} are not assumed.

Performing similar computations as in Subsection~\ref{S3.2}, we obtain the following relative entropy identity associated with the viscous approximations $U^\e$ and the hyperbolic solution $\bar U$:
\begin{align}\label{Sec4.3:relentropypar} 
\del_t(\eta(U^\e|\bar U&;x,t))+\del_\alpha(q_\alpha(U^\e|\bar U,x,t)-\e (G^\e-\bar G)\cdot B_{\alpha\beta}^\e\del_\beta U^\e)+(G^\e-\bar G)\cdot (R^\e -\bar R)  \nonumber\\
+& \e\nabla G^\e\del_\alpha U^\e\cdot B_{\alpha\beta}^\e\del_\beta U^\e=  -\nabla \bar G\del_\alpha \bar U\cdot  f_\alpha(U^\e|\bar U;x,t)
-\bar R \cdot G(U^\e|\bar U;x,t) 
\nonumber\\
& + \eta_t^\e+q_{\alpha,x_\alpha}^\e-\bar \eta_t-\bar q_{\alpha,x_\alpha} -\bar G\cdot(A_t^\e+f_{\alpha,x_\alpha}^\e-\bar A_t-\bar f_{\alpha,x_\alpha}) \nonumber\\
&  -\bar G_t \cdot(A^\e-\bar A) -\bar G_{x_\alpha}\cdot (f_\alpha^\e-\bar f_\alpha)\;\nonumber\\
&+\e \nabla\bar G\del_\alpha\bar U\cdot B_{\alpha\beta}^\e\del_\beta U^\e-\e(G^\e_{x_\alpha}-\bar G_{x_\alpha})\cdot B_{\alpha\beta}^\e\del_\beta U^\e\;.
\end{align}

Here, we prove the convergence of $U^\e$ to $\bar U$ on $\mathbb{R}^d \times [0,T]$ as $\e\to0+$ assuming that such smooth solution $\bar U$ exists. Let us point out that $G^\e$ denotes $G(U^\e,x,t)$, while $\bar G$ denotes $G(\bar U,x,t)$ as before and similarly for the other variables that appear in the computations.
\begin{theorem}
\label{thmconvnew} Let $\bar U(x,t)$ be a Lipschitz solution of the hyperbolic system~\eqref{Sec3:Ueq} defined on a  maximal interval of existence  $ \mathbb{R}^d \times [0, T^*)$ with $T^*\le\infty$ and initial data $U_0$. 
Fix $\e\in(0,\e_0)$ and consider the sequence  smooth solutions $U^\e(x,t)$ of~\eqref{Sec2:Ueqpar}  defined on $\R^d\times[0,T]$, $T<T^*$,  with initial data $U_0^\e$.
Assume that  hypotheses $(H_1)$--\eqref{S2:H3}, \eqref{S2:HB}, \eqref{S2:Hxt}--\eqref{S2:Hxt3}, \eqref{S2:HBxt} hold true together with hypotheses~\eqref{S2:HP2} and~\eqref{S2:HBpar}. If $U^\e$ and $\bar U$ take values in a ball $B_M\subset\R^n$ of radius $M>0$ and $U^\e$, $\del_\alpha U^\e$ and $\bar U$, $\del_\alpha \bar U$ decay to zero
sufficiently fast as $|x|\to\infty$, then there exists a constant 
$C>0$ independent of $\eps$,  such that
\begin{equation}
\label{Sconvergence}
\int_{\mathbb{R}^d} \eta (U^\eps | \bar U;x,t ) \, dx  \le C  \Big (\int_{\mathbb{R}^d} \eta (U_0^\eps | \bar U _0 ;x,0) \, dx  +  \eps  \int_0^T \int_{\mathbb{R}^d} | \nabla \bar U|^2 \, dx ds \Big ) \,,
\end{equation}
for $t\in[0,T]$.
In particular, if $ \int_{\mathbb{R}^d} \eta (U_0^\eps | \bar U _0 ;x,0 ) \, dx \to 0$
and $\nabla \bar U \in L^2 (\mathbb{R}^d\times[0,T])$ 
then 
\begin{equation}
\label{relenconv}
\sup_{ t \in (0, T)} \int_{\mathbb{R}^d} \eta (U^\eps | \bar U;x,t ) ) \, dx  \to 0 \qquad \mbox{as $\eps \to 0$}.
\end{equation}
\end{theorem}
\begin{proof} We integrate~\eqref{Sec4.3:relentropypar} over $\R^d\times[0,T]$ to get
\begin{align}\label{conveqnrelen-2}
\int_{\mathbb{R}^d} \eta (U^\eps& | \bar U;x,t ) \, dx +\eps\int_0^t \int_{\mathbb{R}^d}\big(\nabla G^\eps)^TB_{\alpha\,\beta}^\eps\del_\beta U^\eps\big)\cdot\del_\alpha U^\eps\,dx\nonumber\\
&+\int_0^t\int_{\mathbb{R}^d}(G^\e-\bar G)\cdot (R^\e -\bar R) dxds\nonumber\\
&\le  C  \int_0^t\int_{\mathbb{R}^d}   | f_\alpha (U^\eps | \bar U;x,s ) | \,+| G(U^\e|\bar U;x,s) |dx ds\nonumber\\
&+C\int_0^t\int_{\mathbb{R}^d}  |U^\e-\bar U|^2 dxds + \eps\int_0^t \int_{\mathbb{R}^d} |(\nabla \bar G\del_\alpha \bar U)\cdot\big( B_{\alpha\,\beta}^\e \del_\beta U^\eps\big)|\,dx\nonumber\\
&+\e\int_0^t\int |(G^\e_{x_\alpha}-\bar G_{x_\alpha})\cdot B_{\alpha\beta}^\e\del_\beta U^\e|dxds+\int_{\mathbb{R}^d} \eta (U^\eps_0 | \bar U_0;x,0 ) \, dx
\end{align}
%
using~\eqref{S2:Hxt}--\eqref{S2:Hxt3} and bounds~\eqref{S2:HB} and~\eqref{S2:HBxt}.
The goal is to estimate the $\e$-terms on the right-hand side by the dissipation term. We proceed as follows
\begin{align}\label{conveqnrelen-3}
 \int |(\nabla \bar G\del_\alpha \bar U)\cdot\big( B_{\alpha\,\beta}^\e &\del_\beta U^\eps\big)|\,dx\le\frac{1}{2\lambda_2}\int
  \sum_{i,\alpha}|\nabla\bar {G^i}\del_\alpha\bar U|^2 dx+\frac{\lambda_2}{2}\int\sum_{i,\alpha}|\sum_{j,\beta} B^{\e,i,j}\del_\beta U^{\e,j}|^2|dx\nonumber\\
\le& \frac{C_1}{2\lambda_2}\int| \nabla\bar U|^2 dx+\frac{1}{2}\int (\nabla G^\e\del_\alpha U^\e)\cdot(B_{\alpha\beta}^\e\del_\beta U^\e)dx
\end{align}
%
for some constant $C_1$ that may depend on $M$.
Also by \eqref{S2:HBpar} and using that $\nabla G_{x_\alpha}(\cdot,x,t)$ is bounded, we estimate the other $\e$-term
\begin{align}\label{conveqnrelen-4}
\int |(G^\e_{x_\alpha}-\bar G_{x_\alpha})\cdot B_{\alpha\beta}^\e&\del_\beta U^\e|dx \le\frac{1}{2\lambda_2}\int |G^\e_{x_\alpha}-\bar G_{x_\alpha}|^2 dx+\frac{\lambda_2}{2} \int
\sum_{i,\alpha}|\sum_{j,\beta} B^{\e,i,j}\del_\beta U^{\e,j}|^2|dx\nonumber\\
&\le \frac{C_2}{2\lambda_2}\int|U^\e-\bar U|^2 dx+\frac{1}{2}\int (\nabla G^\e\del_\alpha U^\e)\cdot(B_{\alpha\beta}^\e\del_\beta U^\e)dx\,,
\end{align}
for some constant $C_2>0$. 
Under hypotheses~\eqref{S2:H3} and (H$_1$) and using that $U^\e$ and $\bar U$ take values in $B_M$, the relative entropy is equivalent to the $L^2$ distance, (see~\eqref{S2.3.1L2} in the previous proof). Hence, we have that the relative flux $f_\alpha (U^\e|\bar U;x,t)$, the relative multiplier $G (U^\e|\bar U;x,t)$ are quadratic in $|U^\e-\bar U|$.
Combining~\eqref{conveqnrelen-3}--\eqref{conveqnrelen-4} with~\eqref{conveqnrelen-2}, we arrive at
%
\begin{align}\label{conveqnrelen-5}
\int_{\mathbb{R}^d} \eta (U^\eps | \bar U;x,t ) \, dx \le C_3(1+\e)\int_0^t
\int_{\mathbb{R}^d} \eta (U^\eps | \bar U;x,s ) \, dx ds
+\int_{\mathbb{R}^d} \eta (U^\eps_0 | \bar U_0;x,0 ) \, dx\nonumber\\
+\e \frac{C_1}{2\lambda_2}\int_0^t\int_{\mathbb{R}^d} | \nabla\bar U|^2 dxdt\,,
\end{align}
%
for some positive constant $C_3$.
Then by Gronwall's inequality, we conclude that
\begin{align}
\int_{\mathbb{R}^d} \eta (U^\eps | \bar U;x,t ) \, dx \le  e^{C_3(1+\e)t}(\int_{\mathbb{R}^d} \eta (U^\eps_0 | \bar U_0;x,0 ) \, dx+\e \int_0^T\int_{\mathbb{R}^d} | \nabla\bar U|^2 dxds)\,,
\end{align}
for $t\in[0,T]$ and $\e\in(0,\e_0)$.
%
The proof is complete.
\end{proof}

%
\section{Examples} \label{sec6}
In this section, we provide the reader with examples that belong to the class of hyperbolic inhomogeneous systems of balance laws~\eqref{Sec1:Ueq} and the hyperbolic-parabolic one~\eqref{Sec1:Ueqpar}. These examples are written in one-space dimension and references to the existing theory of weak solutions available for them are given. It would be interesting to examine these systems in several space dimensions in terms of their well-posedness in a weak framework. However, the available results are limited, especially regarding questions of existence of weak solutions. Here, we discuss the set of hypotheses A, B and C in the context of each example and conclude that Theorems~\ref{thmweakstrong},~\ref{thmstability},~\ref{thmconvnew} can apply.

{\bf Example 1}. The system of isentropic gas flow through a duct of (slowly) varying cross section $a(x)$ is
\begin{align}\label{sec6:gasduct}
\begin{split}
&\del_t(a(x)\rho)+\del_x(a(x)\rho v) =0  \\
&\del_t(a(x) \rho v)+\del_x(a(x)\rho v^2+a(x) p(\rho))=a'(x) p(\rho)\;.
\end{split}
\end{align}
This system reduces to the rectilinear isentropic flow of gas when $a(x)$ is constant. One can first verify that the system is strictly hyperbolic with nonzero characteristic speeds and characteristic families that are genuine nonlinear everywhere.  Under the hypothesis that $a(x)$ has sufficiently small total variation on $(-\infty,\infty)$, one can verify the requirements of the next theorem as established by T.-P. Liu and it is taken from Dafermos~\cite[Chap. 16]{MR3468916}.

\begin{theorem}[T.-P. Liu,~\cite{Liu}]
Consider the strictly hyperbolic system of balance laws
\be\label{sec6:U1eq}
\del_t U+\del_xf(U)+P(U,x)=0
\ee
with nonzero characteristic speeds, and characteristic families that are either genuinely nonlinear or linearly degenerate. Assume that for any $U$ in $\mathcal{O}$ and $x\in(-\infty,\infty)$,
\be\label{sec6:U1eqLiu}
|P(U,x)|\le f(x),\qquad |DP(U,x)|\le f(x)
\ee
where $f(x)$ satisfies \be
\int f(x)\,dx\le\omega
\ee
with $\omega$ sufficiently small. If the initial data $U_0$ have bounded variation, with $TV U_0=\delta$ sufficiently small, then there exists a global admissible $BV$ solution $U$ of~\eqref{sec6:U1eq},~\eqref{sec1:Udata}. For each fixed $t\in[0,\infty)$, $U(t)$ is a function of bounded variation on $(-\infty,\infty)$ and
\be
TV_{(-\infty,\infty)} U(t)\le C(\delta+\omega)\,.
\ee
\end{theorem}
Here $\mathcal{O}$ stands for an open subset of $\mathbb{R}^n$ in which $U$ takes values.
Assuming that $a(x)$ remains away from $0$ and invertible, system~\eqref{sec6:gasduct} can be rewritten as
\begin{align}
\begin{split}
&\del_t \rho+\del_x(\rho v) +a^{-1}(x) a'(x)\rho v=0  \\
&\del_t( \rho v)+\del_x(\rho v^2+p(\rho))+a^{-1}(x) a'(x) \rho v^2=0
\end{split}
\end{align}
and this is now in form of system~\eqref{sec6:U1eq}. If $a(x)$ has sufficiently small total variation on $(-\infty,\infty)$, then the result of the above theorem applies. 

System~\eqref{sec6:U1eq} in several space dimensions becomes
\be\label{sec6:U1eq2}
\del_t U+\del_{\alpha}f(U)+P(U,x)=0
\ee
with $x\in\mathbb{R}^d$, and the entropy-entropy flux pair $(\eta(U),q_\alpha(U))$ satisfies 
\be
\del_t\eta(U)+\del_\alpha q_\alpha(U)+\nabla\eta\cdot P\le 0\;.
\ee
It is clear that $A=U$ and $f_\alpha=f_\alpha(U)$ are independent of $x$ and $t$. Hence, hypotheses $(H_1)$-\eqref{S2:H3} correspond to the existence of an entropy-entropy flux pair with a convex entropy $\eta=\eta(U)$. One can directly write from~\eqref{Sec3:relentropyhyp} the relative entropy identity associated with system~\eqref{sec6:U1eq2}
\begin{align}\label{Sec6:relentropyhyp}
\del_t(\eta(U|\bar U))+&\del_\alpha(q_\alpha(U|\bar U))+(\nabla\eta-\nabla\bar \eta)\cdot (P(U,x) -P(\bar U,x)) \le \nonumber\\
&\le  -\nabla^2 \bar \eta\del_\alpha \bar U\cdot  f_\alpha(U|\bar U)-P(\bar U,x) \cdot \nabla\eta(U|\bar U)\;.
\end{align}
Hypothesis~\eqref{S2:HB} holds true since all functions are independent of $x$ and $t$ and the bound $C$ is allowed to depend on $M$. Hypothesis~\eqref{S2:Hxt}--\eqref{S2:Hxt3}, \eqref{S2:Hgr4}--\eqref{S2:Hgr5} are not needed since these terms do not appear in~\eqref{Sec6:relentropyhyp}. Hypotheses~\eqref{S2:HBxt} reduces to $|P(U,x)|\le C$ for $U\in B_M$ which holds true for functions obeying condition~\eqref{sec6:U1eqLiu}. The remaining hypotheses~\eqref{S2:Hgr1}--\eqref{S2:Hgr3},~\eqref{S2:HR1}--\eqref{S2:HG} are the standard growth conditions imposed when proving weak-strong uniqueness in the relative entropy setting and actually, \eqref{S2:HR2} is a condition that corresponds to a dissipative source.

{\bf Example 2}. An important class of media in viscoelasticity in which the flux function depends also on the past history of the material and we say that the material has memory are materials with fading memory. Their constitutive relations can be expressed as systems of the form in one-space dimension
\be\label{S5:ex2:1}
\del_t u+\del_x \left(F(u(x,t))+\int_0^tk(t-\tau) H(u(x,\tau))d\tau\right)=0
\ee
with $F$ and $H$ smooth functions of $u\in\mathbb{R}^n$ and $k$ a smooth kernel integrable over $\mathbb{R}_+=[0,\infty)$. Such systems have the property that smooth initial data near equilibrium generate globally smooth solutions, Renardy, Hrusa, and Nohel~\cite{RHN} in contrast to the situation with elastic media in which classical solutions in general break down in finite time even when the initial data is small. 
While when the initial data are ``large", the destabilizing action of nonlinearity of the flux function  prevails over the damping, and solutions break down in a finite time; see Dafermos~\cite{Daf86} and Malek-Madani and Nohel~\cite{MN}. Special forms of~\eqref{S5:ex2:1}, such as the scalar equation
\be\label{S5:ex2:2}
\del_t u+\del_x f(u)+\int_{0}^t k(t-\tau)\del_x f(u(x,\tau)\,d\tau=0
\ee
or the system
\begin{align}\label{S5:ex2:3}
\begin{split}
&\del_t  u+\del_x v=0 \\
&\del_t v+\del_x p(u)+\int_{-\infty}^t k(t-\tau)\del_x q(u(x,\tau)\,d\tau=0
\end{split}
\end{align}
with a relaxation term, that both capture the damping effect of memory have been studied via the vanishing viscosity approximation (cf.~\cite{CC1},~\cite{CC2}).  The main motivation for the analysis in these works is that such models can be viewed as a linear Volterra equation, which was first observed by MacCamy~\cite{Mac} and later employed in Dafermos~\cite{Daf88} and Nohel, Rogers, and Tzavaras~\cite{NRT}.
Hence, under appropriate conditions on $k$ and applying the theory of resolvent kernel $r(\cdot)$ associated to $k(\cdot)$ and redistribution of damping, equation~\eqref{S5:ex2:2} is equivalent to
\be\label{S5:ex2:4}
\del_t u+\del_x f(u)+r(0) u= r(t) u- \int_0^t r'(t-\tau) u(\tau,x)\,d\tau
\ee
and system~\eqref{S5:ex2:3} is equivalent to
\begin{align}\label{S5:ex2:5}
\begin{split}
&\del_t  u+\del_x v +r(0) u=0 \\
&\del_t v+\del_x p(u)+r(0) v= r(t) v-\int_0^t r'(t-\tau) v(\tau,x)\,d\tau
\end{split}
\end{align}
The resolvent kernel $r$ is a nonincreasing positive kernel in $L^1$ and this yields a dissipative source in~\eqref{S5:ex2:4} and~\eqref{S5:ex2:5}.

Inhomogeneiry is thus present only in the source $P=P(u,t)$. Hence, hypotheses $(H_1)$-\eqref{S2:H3} correspond to the existence of an entropy-entropy flux pair with a convex entropy $\eta=\eta(U)$ and the relative entropy identity associated with~\eqref{S5:ex2:4} or~\eqref{S5:ex2:5}  (in d-dimensions) takes the form 
\begin{align}\label{Sec6:relentropyhyp2}
\del_t(\eta(U|\bar U))+&\del_\alpha(q_\alpha(U|\bar U))+(\nabla\eta-\nabla\bar \eta)\cdot (P(U,t) -P(\bar U,t)) \le \nonumber\\
&\le  -\nabla^2 \bar \eta\del_\alpha \bar U\cdot  f_\alpha(U|\bar U)-P(\bar U,t) \cdot \nabla\eta(U|\bar U)\;.
\end{align}
Hypotheses~\eqref{S2:HB} holds true immediately and hypotheses~\eqref{S2:Hxt}--\eqref{S2:Hxt3}, \eqref{S2:Hgr4}--\eqref{S2:Hgr5} are not needed. However, hypothesis~\eqref{S2:HBxt} becomes the condition $|P(u,t)|\le C$, which in general true when $C=C(T)$. Let us note that one should take advantage of the properties of the kernel $r$ that correspond to a dissipative source $P$ and the term on the left-hand side of~\eqref{Sec6:relentropyhyp2} could act as the dissipative term.

{\bf Example 3}. The last example considers self-similar viscous limits that are obtained as $\e\to 0$ in the sequence of viscous solution $u^\e$ to systems
%
\begin{align}\label{S5:ex3:1}
\begin{split}
&\del_t  u+\del_x f(u)=\e\,t\del_x(\widetilde{B}(u)\del_x u)
\end{split}
\end{align}
written here in one-space dimension ($x\in\mathbb{R}$). Solutions to~\eqref{S5:ex3:1} have been studied in the setting of the Cauchy problem and IBVP  with Riemann data (cf.~\cite{Daf2}, ~\cite{Tz1},~\cite{CS1}, \cite{CS2} and the references therin). Existence of $u^\e$ and convergence in the framework of bounded variation is established for a viscosity matrix $\widetilde{B}=I$. For more general viscosity matrices, not necessarily invertible, the characterization of boundary layers is provided and these remain unchanged if one considers the standard viscous limits $\e\del_x^2u$ instead of~\eqref{S5:ex3:1}. Moreover, the description of the Riemann solution to~\eqref{S5:ex3:1} using the center manifold techniques is analyzed in Dafermos~\cite{MR3468916}. Such systems are investigated in several-space dimensions in an effort to find a selection criterion for weak solutions in order to single out a unique one. This idea is discussed in the analysis by Giesselmann--Tzavaras~\cite{GT}. System~\eqref{S5:ex3:1} generalized in several-space dimensions takes the form
\begin{align}\label{S5:ex3:2}
\begin{split}
&\del_t  u+\del_\alpha f_\alpha(u)=\e\,t\del_{\alpha}(\widetilde{B}_{\alpha\beta}(u)\del_\beta u)
\end{split}
\end{align}
and the corresponding relative entropy identity becomes
\begin{align}\label{Sec5.3:relentropypar} 
\del_t(\eta(u|\bar u))+\del_\alpha(q_\alpha(u|\bar u)-\e j_\alpha) +\e D 
\le  -\nabla^2 \bar\eta \del_\alpha \bar U\cdot  f_\alpha(U|\bar U;x,t)
 +\e Q\;,
\end{align}
with
\begin{align}
j_\alpha:
&= t(\nabla\eta - \nabla\bar \eta) \cdot ( \widetilde{B}_{\alpha \beta} \del_\beta U - \bar {\widetilde{B}}_{\alpha \beta} \del_\beta \bar U ) + 
t\bar {\widetilde{B}}_{\alpha \beta}   \del_\beta \bar U \cdot \nabla\eta(U | \bar U;x,t)\nonumber
\end{align}
and $Q=\sum_{i=2}^6 Q_i$, given in~\eqref{definq2}--\eqref{definq6}, by replacing $B_{\alpha\beta}(U,x,t)$ by $t\cdot \widetilde{B}_{\alpha\beta}(u)$. Inhomogeneity is not present in the conserved part of~\eqref{S5:ex3:2}, hence hypotheses~\eqref{S2:HB},~\eqref{S2:Hxt}--\eqref{S2:Hxt3},~\eqref{S2:HBxt},~\eqref{S2:Hgr4}-\eqref{S2:Hgr5},~\eqref{S2:HR1}--\eqref{S2:HG} are not needed. Also, hypothesis \eqref{S2:HBpar} holds true with $C=C'T$, where $T$ is the final time of existence. Regarding hypotheses~\eqref{S2:HP1} or~\eqref{S2:HP2}, we see that $\lambda_i=\tilde{\lambda}_i\cdot t$ induces a singularity at time zero and the constants that bound the $\e$-terms in the proofs of Theorems~\ref{thmstability} and~\ref{thmconvnew} depend on $(\tilde{\lambda}_i t)^{-1}$. Let us only here make two comments:
(1) To treat the stability result of Section~\ref{S4.2}, , see that~\eqref{Sec4.2:relentropypar2} takes the form
\begin{align}\label{Sec5.3:relentropypar2} 
\|u(t)-\bar u(t)\|^2\le \|u_0-\bar u_0\|^2+2C(1+\frac{\e}{\tilde{\lambda}_1 t})\int_0^t\|u(s)-\bar u(s)\|^2ds
\end{align}
One could overcome such a difficulty by analyzing system~\eqref{S5:ex3:2} in its natural framework which is the self-similar coordinates $\xi_\alpha:=\frac{x_\alpha}{t}$. (2) To deal with the convergence result of Section~\ref{S4.3}, estimate~\eqref{conveqnrelen-5} takes the form
\begin{align}\label{S5.3:conveqnrelen-5}
\|u^\e(t)-\bar u(t)\|^2 \le C_2\int_0^t
\|u^\e(s)-\bar u(s)\|^2ds 
+\|u_0-\bar u_0\|^2
+\e \frac{C_1}{2t \lambda_2}\int_0^T\| \nabla\bar U\|^2dt\,,
\end{align}
and note that estimate~\eqref{conveqnrelen-4} is not needed since such a term is not present here. This means that  the factor $C_2$ of $\|u^\e(s)-\bar u(s)\|^2$ above  is not singular.
One could integrate over $(t_1,t)$, for fixed $t_1>0$, to avoid the origin, then pass to $\e\to0+$ and last take the limit $t_1\to0$. This analysis is not the scope of this article.

In a forthcoming paper, the author presents applications in the context of isometric immersion to the Gauss--Codazzi system and establishes the uniqueness result for smooth immersions within the class of corrugated immersions. Let us only mention here that this forthcoming work investigates the Gauss-Codazzi system that takes the form~\eqref{Sec1:Ueq} and various cases are exploited according to the given metric. Let us clarify that for the isometric immersion problem, the presence of inhomogeneity may varies depending on the given metric and appropriate change of variables. Therefore, this study is postponed in a forthcoming paper where extensive and detailed analysis is exploited. Existence of $C^{1,1}$ corrugated immersions using techniques from continuum physics was established for different cases and an exposition of the current state can be found in~\cite{CSEMAI}.

	\section*{Acknowledgments}
	The author was partially supported by the Internal grant SBLawsMechGeom \#21036 from University of Cyprus.

\end{document}